\documentclass[11pt,reqno]{amsart}

\usepackage{amsmath,amssymb,amsfonts,
caption,latexsym,graphicx,multirow,
color,hyperref,enumerate}

\newcommand{\A}{\mathrm{A}} \newcommand{\AGL}{\mathrm{AGL}} \newcommand{\AGaL}{\mathrm{A\Gamma L}}  \newcommand{\Aut}{\mathrm{Aut}}
\newcommand{\B}{\mathrm{B}}  
\newcommand{\C}{\mathrm{C}}   \newcommand{\calC}{\mathcal{C}}     \newcommand{\Co}{\mathrm{Co}} 
\newcommand{\D}{\mathrm{D}} 
\newcommand{\E}{\mathrm{E}}
\newcommand{\F}{\mathrm{F}} \newcommand{\Fi}{\mathrm{Fi}}
\newcommand{\G}{\mathrm{G}}     \newcommand{\GL}{\mathrm{GL}} \newcommand{\GO}{\mathrm{O}} \newcommand{\GU}{\mathrm{GU}}
\newcommand{\He}{\mathrm{He}}  \newcommand{\HS}{\mathrm{HS}}
\newcommand{\J}{\mathrm{J}}

\newcommand{\lefthat}{\scalebox{1.3}[1]{\text{$\hat{~}$}}}
\newcommand{\M}{\mathrm{M}} 

\newcommand{\Nor}{\mathbf{N}}
\newcommand{\Out}{\mathrm{Out}}
  \newcommand{\PGL}{\mathrm{PGL}}   \newcommand{\PGU}{\mathrm{PGU}} \newcommand{\ppd}{\mathrm{ppd}} \newcommand{\POm}{\mathrm{P\Omega}} \newcommand{\PSL}{\mathrm{PSL}}   \newcommand{\PSO}{\mathrm{PSO}} \newcommand{\PSp}{\mathrm{PSp}} \newcommand{\PSU}{\mathrm{PSU}}
\newcommand{\Q}{\mathrm{Q}}
 \newcommand{\Rad}{\mathbf{R}}  \newcommand{\Ru}{\mathrm{Ru}}
 \newcommand{\SL}{\mathrm{SL}} \newcommand{\SO}{\mathrm{SO}} \newcommand{\Soc}{\mathrm{Soc}} \newcommand{\Sp}{\mathrm{Sp}}  \newcommand{\SU}{\mathrm{SU}} \newcommand{\Suz}{\mathrm{Suz}} \newcommand{\Sy}{\mathrm{S}} \newcommand{\Sym}{\mathrm{Sym}} 

\newcommand{\Z}{\mathbf{Z}} \newcommand{\ZZ}{\mathrm{C}}

\newtheorem{theorem}{Theorem}[section]
\newtheorem{lemma}[theorem]{Lemma}

\theoremstyle{definition}

\newtheorem{hypothesis}[theorem]{Hypothesis}

\newtheorem{question}[theorem]{Qestion}

\newcommand{\Ga}{\Gamma}  

\newcommand{\bO}{\mathbf{O}}

\renewcommand{\mod}{\mathrm{mod}\,}
\newcommand{\Th}{\mathrm{Th}}
\newcommand{\McL}{\mathrm{McL}}
\newcommand{\Ly}{\mathrm{Ly}}
\newcommand{\ON}{\mathrm{O'N}}
\newcommand{\HN}{\mathrm{HN}}

\numberwithin{equation}{section}

\begin{document}

\title[$2$-Arc-transitive digraph]{The smallest vertex-primitive $2$-arc-transitive digraph}

\author[Yin]{Fu-Gang Yin}
\address{Fu-Gang Yin\\
School of Mathematics and Statistics\\
Central South University\\
Changsha 410083, Hunan, P.R. China.}
\address{Department of Mathematics\\
Beijing Jiaotong University\\
Beijing, 100044, People's Republic of China}

\email{18118010@bjtu.edu.cn}

\author[Feng]{Yan-Quan Feng}
\address{Y.-Q. Feng\\
Department of Mathematics\\
Beijing Jiaotong University\\
Beijing, 100044, People's Republic of China}
\email{yqfeng@bjtu.edu.cn}

\author[Xia]{Binzhou Xia}
\address{B. Xia\\
School of Mathematics and Statistics\\
The University of Melbourne\\
Parkville, VIC 3010\\
Australia}
\email{binzhoux@unimelb.edu.au}

\begin{abstract}
In 2017, Giudici, Li and the third author constructed the first known family of vertex-primitive $2$-arc-transitive digraphs of valency at least $2$. The smallest digraph in this family admits $\mathrm{PSL}_3(49)$ acting $2$-arc-transitively with vertex-stabilizer $\mathrm{A}_6$ and hence has $30758154560$ vertices. In this paper, we prove that this digraph is the vertex-primitive $2$-arc-transitive digraph of valency at least $2$ with fewest vertices.


\end{abstract}

\maketitle

\section{Introduction}

In this paper, a digraph $\Ga$ is a pair $(V, \to)$ with a set $V$ (of vertices) and an antisymmetric irreflexive binary relation $\to$ on $V$.
For a positive integer $s$, an \emph{$s$-arc} of $\Ga$ is a sequence $(v_0,v_1, \dots, v_s)$ of vertices with $v_i\rightarrow v_{i+1}$
for each $i\in\{0,1, \dots, s-1\}$.
A $1$-arc is  simply called an \emph{arc}.
We say that $\Ga$ is \emph{$s$-arc-transitive} if the  automorphism group $\Aut(\Ga)$  of $\Ga$ acts transitively on the set of $s$-arcs.
A permutation group  $G$ on a set $\Delta$ is said to be \emph{primitive} if $G$ does not preserve any nontrivial partition of $\Delta$.
We say that $\Ga$ is \emph{vertex-primitive} if $\Aut(\Ga)$ acts primitively on the vertex set $V(\Ga)$ of $\Ga$.

A digraph $\Ga$ with arc set $A(\Ga)$ is said to be \emph{regular} of \emph{valency} $d$ if both the set $\Ga^-(v):=\{ u \in V\mid (u,v) \in A(\Ga)\}$ of in-neighbors of $v$ and the set $\Ga^+(v):= \{ w \in V\mid (v,w) \in A(\Ga)\}$ of out-neighbors of $v$ have size $d$ for all $v \in V $.
Observe that a regular $(s+1)$-arc-transitive digraph is necessarily $s$-arc-transitive.
For each $s\geq 1$ and $d\geq1$, there exist infinitely many $s$-arc-transitive digraphs of valency $d$, see Praeger~\cite{P1989}.
However, vertex-primitive $2$-arc-transitive digraphs turn out to be rare.
In fact, since asked by Praeger~\cite[Question~5.9]{P1990} in 1990, the existence of vertex-primitive $2$-arc-transitive digraphs of valency at least $2$ was open for nearly $30$ years until 2017, when Giudici, Li and third author~\cite{GLX2017} constructed the first family of such digraphs.

The digraphs constructed by Giudici, Li and third author, denoted by $\Ga_p$, have $|\PSL_3(p^2)|/|\A_6|$ vertices (see~\cite[Theorem~1.1]{GLX2017}), where $p\geq7$ is prime such that $p\equiv \pm 2\pmod{5}$. Hence the smallest one $\Ga_7$ in this family has order $|\PSL_3(49)|/|\A_6|=30758154560$. The main purpose of this paper is to show that $\Ga_7$ is the smallest among all vertex-primitive $2$-arc-transitive digraphs of valency at least $2$.

\begin{theorem}\label{th:main}
Let $\Gamma$ be a vertex-primitive $2$-arc-transitive digraph of valency at least $2$.
Then $V(\Ga)\geq 30758154560$. Moreover, if $V(\Ga)= 30758154560$, then $\Ga\cong\Ga_7$.
\end{theorem}

In~\cite{GX2017}, Guidici  and the third author asked the following question:

\begin{question}\label{question:1}
Is there an upper bound on $s$ for vertex-primitive $s$-arc-transitive digraphs  that are not directed cycles?
\end{question}

Moreover, in~\cite{GX2017}, the above question was reduced to that of almost simple groups, which has been studied in a few recent papers~\cite{CGP2021+,CLX2021+,GLX2019,PWY2020}. We believe that some ideas and results in our proof of Theorem~\ref{th:main} would be helpful for studying Question~\ref{question:1}. For instance, in Lemma~\ref{lem4.4}, we actually show that $s\leq 1$ if the socle is a sporadic simple group not isomorphic to $\mathbb{M}$, $\mathbb{B}$, $\Fi_{24}'$ or $\Co_1$.

We thank Professor Jin-Xin Zhou and Wenying Zhu for reading the first draft of this paper and making valuable comments.
We are also grateful to the anonymous referee for helpful suggestions to improve the paper.
We acknowledge {\sc Magma}~\cite{Magma} and the support and resources from the Center for High Performance Computing at Beijing Jiaotong University for computation.
This work was supported by the National Natural Science Foundation of China (12161141005, 12271024, 12071023,11971054), the 111 Project of China (B16002) and Yunnan Applied Basic Research Projects (202101AT070137).

\section{Preliminaries}

For a finite group $G$, denote by $\Soc(G)$ the socle of $G$, by $\Rad(G)$ the largest solvable normal subgroup of $G$, by $G^{(\infty)}$ the
smallest normal subgroup of $G$ such that $G/G^{(\infty)}$ is solvable, and by $\bO_p(G)$ the largest normal $p$-subgroup of $G$, for each prime $p$.
For $H\leq G$ and $g\in G$, let $H^g:=g^{-1}Hg$ denote the conjugate of $H$ under $g$.
The subgroup $H$ is said to be \emph{core-free} in $G$ if $H$ contains no nontrivial normal subgroup of $G$.
For a almost simple group $G$, denote by $P(G)$ the minimal index of core-free subgroups in $G$.
Note that $P(G)\geq P(\Soc(G))$.

For a nonzero integer $n$ and prime number $r$, denote by $n_r$ the $r$-part of $n$ (that is, the largest power of $r$
dividing $n$), and $\pi(n)$ the set of prime divisors of $n$. If $G$ is a group, then $\pi(G):=\pi(|G|)$.

Given integers $a\geq 2$ and  $m\geq 2$, a prime number $r$ is called a primitive prime divisor of the pair $(a, m)$
if $r$ divides $a^m - 1$ but does not divide $a^i -1 $ for any positive integer $i <m$. By a  theorem of Zsigmondy (see for example~\cite[Theorem IX.8.3]{BH1982}), $(a, m)$ always has a primitive prime divisor except when $(a, m) =(2, 6)$ or $a +1$ is a power of $2$ and $m =2$. Denote the set of primitive prime divisors of $(a, m)$ by $\ppd(a, m)$ if $(a, m) \neq (2, 6)$, and set $\ppd(2, 6) = \{7\}$. Note that for each $r \in \ppd(a, m)$ we have $r \equiv 1(\mod m)$ and so $r> m$.

For a digraph $\Ga$ and a subgroup $G$ of $\Aut(\Ga)$, the digraph $\Ga$ is said to be \emph{$G$-vertex-primitive} or \emph{$(G,s)$-arc-transitive}, respectively, if $G$ acts primitively on the vertex set $V(\Ga)$ of $\Ga$ or transitively on the set of $s$-arcs of $\Ga$.

\subsection{Arc-transitive digraph}
\ \vspace{1mm}

By definition, a $(G,2)$-arc-transitive digraph is clearly $G$-arc-transitive. Moreover, the next lemma, which is a special case of~\cite[Corrollary 2.11]{GX2017}, shows that a $(G,2)$-arc-transitive digraph is $M$-arc-transitive for every vertex-transitive normal subgroup $M$ of $G$.

\begin{lemma}\label{pro:Tarctrans}
 Let $\Ga$ be a $(G,2)$-arc-transitive digraph, and let $M$ be a vertex-transitive normal subgroup of $G$.
Then  $\Ga$ is $M$-arc-transitive.
\end{lemma}

In the rest of this subsection, we collect some results on arc-transitive digraphs.

\begin{lemma}[{\cite[Lemma~2.13]{GLX2019}}]\label{pro:valency}
For each vertex-primitive arc-transitive digraph $\Ga$, either $\Ga$ is a directed cycle of prime length
or $\Ga$ has valency at least $3$.
\end{lemma}

The following result is a consequence of the connectivity of arc-transitive digraphs.
It will be used repeatedly in this paper.

\begin{lemma}[{\cite[Lemma 2.14]{GLX2019}}]\label{pro:Gvnormal}
Let $\Ga$ be a connected $G$-arc-transitive digraph with an arc $(v,w)$, and let $g \in G$ such that $v^g = w$.
Then no nontrivial normal subgroup of $G_v$ is normalized by $g$.
\end{lemma}

Let $G$ be a transitive permutation group on a set $\Omega$.
Then $G$ acts naturally on $\Omega\times\Omega$.
% \[
% \Omega^{(2)}:=\{(v,w)\mid v,w\in\Omega\}\setminus\{(v,v)\mid v\in\Omega\}.
% \]
A  $G$-orbit  in $\Omega\times\Omega$ is called a \emph{$G$-orbital}, and a $G$-orbital $(v,w)^G$  is said to be \emph{self-paired}  if $(w,v)=(v,w)^g$ for some $g\in G$.
A $G_v$-orbit in $\Omega$ is called a \emph{$G$-suborbit} relative to $v$.
Since $G$ is transitive, there is a bijection between $G$-orbitals and  $G_v$-orbits in $\Omega$.
We say a $G$-suborbit is self-paired if the corresponding $G$-orbital is self-paired, and non-self-paired otherwise.
%The number of $G_v$-orbits is said to be the \emph{rank} of $G$, and the length of a $G_v$-orbit  is called a \emph{subdegree} of $G$.
For a non-self-paired $G$-orbital $(v,w)^G$, the digraph with vertex set $\Omega$ and arc-set $(v,w)^G$ is an arc-transitive digraph.
Conversely, each arc-transitive digraph arises in this way.

Let  $[G\,{:}\,G_v]$ be the set of right cosets of $G_v$ in $G$.
Then the action of $G$ on $\Omega$ is equivalent to the action of $G$ on $[G\,{:}\,G_v]$ by right multiplication.
Note that in the latter action, each $G$-suborbit is an $(G_v,G_v)$-double coset $G_vgG_v$, and it is non-self-paired if and only if $g^{-1} \notin G_vgG_v$.
To summarize, we have the following observation.

\begin{lemma}\label{lm:orbital}
Let $\Ga$ be a $G$-arc-transitive digraph with an arc $(v,w)$.
Let $g \in G$ such that $v^g = w$.
Then $w$ lies in a non-self-paired $G$-suborbit and $g^{-1} \notin G_vgG_v$.
\end{lemma}

%\begin{proposition}\label{pro:cosetgraph}
%Let $\Ga=\Cos(G,H,HgH)$.
%%In the above notation, the following hold.
%\begin{enumerate}[\rm (a)]
%\item $\Ga$ is $|H: H \cap H^g|$-regular.
%\item $\Ga$ is connected if and only if $\langle
%H,g\rangle =G$
%\item  $R_H (G)$ is primitive on $V(\Ga)$ if and only if $H$ is maximal in $G$.
% \item $R_H(G)$ acts transitively on the set of $2$-arcs of $\Ga$ if and only if
% \begin{equation}\label{eq:Hg}
% H=(H \cap H^g)(H^{g^{-1}} \cap H).
% \end{equation}
%\end{enumerate}
%\end{proposition}

%We close this subsection with an observation that will be used repeatedly throughout the paper.

\subsection{Group factorization}
\ \vspace{1mm}

If a group $G$ is expressed as the product of two subgroups $A$ and $B$, then the expression  $G=AB$ is called a \emph{factorization} of $G$, where $A$ and $B$ are called \emph{factors}.
The following lemma gives some simple facts about factorization of group.

\begin{lemma} \label{lm:basicfacs}
Let $A$ and $B$ be subgroups of $G$.
Then the following are equivalent:
\begin{enumerate}[\rm (a)]
\item $G=AB$;
\item $G=BA$;
\item $G=A^xB^y$ for any $x,y \in G$;
\item $|A\cap B||G|=|A||B|$;
\item $A$ acts transitively by right multiplication on the set of right cosets of $B$ in $G$;
\item $B$ acts transitively by right multiplication on the set of right cosets of $A$ in $G$.
\end{enumerate}
\end{lemma}

For convenience, we define two sets of simple groups throughout this paper:
\begin{equation}\label{eq:T1T2}
\begin{split}
&\mathcal{T}_1:=\{\A_6,\M_{12},\Sp_4(2^f),\POm^{+}_8(q) \},\\
&\mathcal{T}_2:=\{  \PSL_2(q),\PSL_3(3),\PSL_3(4),\PSL_3(8),\PSU_3(8),\PSU_4(2)\}.
\end{split}
\end{equation}

\begin{lemma}\label{lm:ASfac}
Let $H$ be an almost simple group with socle $M$.
Suppose $H = KL$ with nonsolvable core-free subgroups $K$ and $L$ such that $K$ and $L$ have the same nonsolvable composition factors and the same multiplicities.
Then $M=(K\cap M)(L\cap M)$ with $M\in \mathcal{T}_1$, and interchanging $K$ and $L$ if necessary, one of the following holds:
\begin{enumerate}[\rm (a)]
\item $M=\A_6$, and $(H,K,L)\cong (\A_6,\A_5,\A_5)$ or $(\Sy_6,\Sy_5,\Sy_5)$;
\item $M=\M_{12}$, and $(H,K,L)\cong (\M_{12},\M_{11},\M_{11})$;
 %\item $T=\Sp_4(q)$ with $q\geq 4$  even, $G\leq  T.\langle \phi \rangle$ where $\phi$ is a field automorphism as in~\cite[Table 8.14]{BHRD2013}, and $(A\cap T,B\cap T)\cong(\Sp_2(q^2).2,\Sp_2(q^2).2)$ or $(\Sp_2(q^2).2,\Sp_2(q^2))$.
\item $M=\Sp_4(q)$ with $q\geq 4$  even, $H\leq\mathrm{\Gamma Sp}_4(q)$, and $(K\cap M,L\cap M)\cong(\Sp_2(q^2).2,\Sp_2(q^2))$ or $(\Sp_2(q^2).2,\Sp_2(q^2).2)$;
%\item $T=\POm_8^{+}(q)$, $G\leq  T.\langle \gamma, \delta',\phi \rangle $ if $q$ is odd, or $G\leq  T.\langle \gamma, \phi \rangle $ if $q$ is even, where $\phi$ is a field automorphism, $ \delta'$ is a diagonal automorphism induced by $\SO_{8}^{+}(q)\setminus \Omega_{8}^{+}(q)$ and $\gamma$ is graph automorphism of order $2$, as  in~\cite[Table 8.50]{BHRD2013}, and $(A\cap T,B\cap T)\cong (\POm_7(q),\POm_7(q))$.
\item $M=\POm_8^{+}(q)$, $H\leq  \mathrm{P\Gamma O}_8^{+}(q) $, and $(K\cap M,L\cap M)\cong (\Omega_7(q),\Omega_7(q))$.
\end{enumerate}
\end{lemma}

\begin{proof}
The case that both $K$ and $L$ have at least two nonsolvable composition factors (counted with multiplicities) is ruled out by~\cite[Theorem 1.1]{LX2019}.
Thus $K$ and $L$ have only one nonsolvable composition factor.
By~\cite[Theorem 1.4]{LWX2021+}, the triple $(M,K^{(\infty)},L^{(\infty)})$ is one of the following:
\begin{align*}
&(\A_6,\A_5,\A_5),\ (\M_{12},\M_{11},\M_{11}),\ (\POm_8^{+}(q),\Omega_7(q),\Omega_7(q)),\\
&(\Sp_4(q), \Sp_2(q^2),\Sp_2(q^2))\text{ with }q\geq 4\text{ even}.
\end{align*}
In particular, $\pi(M)=\pi(K)=\pi(L)$.
Such factorizations $H=KL$ are listed in~\cite[Table I]{BP1998}, which leads to the conclusion of the lemma.
\end{proof}

\begin{lemma}[{\rm\cite[Proposition 4.1]{LX2021+}}]\label{lm:ASfac2}
Let $H$ be an almost simple group with socle $M$.
Suppose $H= KL$ with $K$ and $L$ solvable.
Then $M\in \mathcal{T}_2$, and interchanging $K$ and $L$ if necessary, one of the following holds:
\begin{enumerate}[\rm (a)]
\item $M=\PSL_2(q)$, $K\cap M\leq \D_{2(q+1)/d}$ and $ L\cap M \leq q{:}((q-1)/d)$, where $d=(2,q-1)$.
\item $(H,K,L)$ satisfies Table \ref{tb:exceptionalfacs}, where $\mathcal{O}\leq \ZZ_2$, and $\mathcal{O}_1$ and $\mathcal{O}_2$ are subgroups of $\mathcal{O}$ such that $\mathcal{O}=\mathcal{O}_1\mathcal{O}_2$.
\end{enumerate}

\begin{table}[h]
\caption{Exceptional factorizations with solvable factors}\label{tb:exceptionalfacs}
\[
\begin{array}{lll}\hline
H&K&L\\\hline
\PSL_2(7).\mathcal{O} & 7, 7{:}(3\times \mathcal{O}) &\Sy_4\\
\PSL_2(11).\mathcal{O} & 11, 11{:}(5\times \mathcal{O}) &\A_4\\
\PSL_2(23).\mathcal{O} &  23{:}(3\times \mathcal{O}) &\Sy_4\\
\PSL_3(3).\mathcal{O} &  13,13{:}(3\times \mathcal{O}) &3^2{:}2.\Sy_4\\
\PSL_3(3).\mathcal{O} & 13{:}(3\times \mathcal{O})&\AGaL_1(9)\\
\PSL_3(4).(\Sy_3 \times\mathcal{O}) & 7{:}(3\times \mathcal{O}).\Sy_3 &2^4{:}(3\times \D_{10}).2\\
\PSL_3(8).(3 \times\mathcal{O}) &   57{:}(9 \times \mathcal{O}_1) &2^{3+6}{:}7^2{:}(3 \times \mathcal{O}_2 )\\
\PSU_3(8).3^2.\mathcal{O} &57{:}(9\times\mathcal{O}_1 ) &2^{3+6}{:}(63{:} 3 ).\mathcal{O}_2 )\\
\PSU_4(2).\mathcal{O} &2^4{:}5&3_{+}^{1+2}{:}2.(\A_4.\mathcal{O})\\
\PSU_4(2).\mathcal{O} &2^4{:}5&3_{+}^{1+2}{:}2.(\A_4.\mathcal{O})\\
\PSU_4(2).\mathcal{O} &2^4{:}{\D_{10}}.\mathcal{O}_1&3_{+}^{1+2}{:}2.(\A_4.\mathcal{O}_2)\\
\PSU_4(2).2 &2^4{:}5{:}4&3_{+}^{1+2}{:}\Sy_3,3^3{:}(\Sy_3\times \mathcal{O}),\\
&&3^3{:}(\A_4\times 2),3^3{:}(\Sy_4\times \mathcal{O})\\
\M_{11}&11{:}5&\M_{9}.2\\\hline
\end{array}
\]
\end{table}
\end{lemma}

A factorization $G=AB$ with $A\cong B$ is said to be \emph{homogeneous}.
A special case of~\cite[Lemma~2.2]{GX2017} gives an important characterization of $2$-arc-transitive digraphs as follows.

\begin{lemma}\label{pro:Gvfactorization}
Let $\Ga$ be a $G$-arc-transitive digraph with a $2$-arc $(u,v,w)$.
Then  $\Ga$ is $(G,2)$-arc-transitive if and only if $G_v=G_{uv}G_{vw}$.
\end{lemma}

Note from the arc-transitivity of $G$ that the factorization $G_v=G_{uv}G_{vw}$ in Lemma~\ref{pro:Gvfactorization} is homogeneous.

\subsection{Suborbits of three classes of primitive groups} \label{subsec:PrimGroups}
\ \vspace{1mm}

In this subsection, we discuss suborbits of three classes of primitive groups.
The first is $\Sy_n$ or $\A_n$ with point stabilizer an imprimitive group.

Let $ G=\Sy_n$ and $\Delta$ be the set of partitions of $\{1,2,\dots,n\}$ with $k>1$ blocks of size $m>1$ (so that $n=mk$).
Then $G$ acts on $\Delta$ naturally and primitively (see~\cite{LPS1987}), and the stabilizer in $G$ is an imprimitive wreath product group $\Sy_m\wr\Sy_k$.
For $v,w\in\Delta$, where $v=\{V_1,V_2,\dots,V_k \}$ and $w=\{W_1,W_2,\dots,W_k \}$, we have
\begin{align*}
V_i=\bigcup_{j=1}^k(V_i \cap W_j)&\text{ and }\sum_{j=1}^k|V_i \cap W_j|=m\text{ for all }i\in\{1,\dots,k\},\\
W_j=\bigcup_{i=1}^k(V_i \cap W_j)&\text{ and }\sum_{i=1}^k|V_i \cap W_j|=m\text{ for all }j\in\{1,\dots,k\}.
\end{align*}
We call the matrix $(|V_i \cap W_j|)_{k\times k}$ an \emph{intersection matrix} of $v$ and $w$.
Note that we may obtain different intersection matrices of $v$ and $w$ by changing the orders of $V_i$ and $W_j$, but all the intersection matrices of $v$ and $w$ form the set
\[
\{P(|V_i \cap W_j|)_{k\times k}Q\mid P,Q\text{ are $k\times k$ permutation matrices}\}.
\]
 % (Recall that a $k\times k$ \emph{permutation matrix} is a matrix obtained by permuting the rows of a $k\times k$ identity matrix according to some permutation on $\{1,2,\dots,k\}$.)

\begin{lemma}\label{lm:imprimitiveaction}
Let $ G=\Sy_n$, let $T=\Soc(G)=\A_n$, and let $\Delta$ be the set of partitions of $\{1,2,\dots,n\}$ with $k>1$ blocks of size $m>1$.
Take $v=\{V_1,\dots,V_k\},w=\{W_1,\dots,W_k\}$ and $u=\{U_1,\dots,U_k\}$ from $\Delta$, and let $A=(|V_i \cap W_j|)_{k\times k}$ and $B=(|V_i \cap U_j|)_{k\times k}$.
Then
\begin{enumerate}[\rm (a)]
\item $w\in u^{G_v}$ if and only if there exist permutation matrices $P$ and $Q$ such that $B=PAQ$.
\item $v$ and $w$ are interchanged by some element in $G$ if and only if there exist permutation matrices $P$ and $Q$ such that $A^{\mathsf{T}}=PAQ$, where $A^{\mathsf{T}}$ is the transpose of $A$.
\item Suppose that $v$ and $w$ are interchanged by some element in $G\setminus T$.
Then $v$ and $w$ are interchanged by some element in $T$ if and only if  $G_{vw}\nleq T$.
\item If $v$ and $w$ are interchanged some element in $G\setminus T$ but no element in $T$, then $|V_i \cap W_j|\leq1$ for all $i,j\in\{1,\dots,k\}$, and in particular, $k\geq m$.
\end{enumerate}
\end{lemma}

\begin{proof}
We first prove (a).
Suppose $w\in u^{G_v}$, that is, $w=u^g$ for some $g\in {G_v}$.
Then $ \{ W_1 ,W_2 ,\dots,W_k \} =\{ U_1^g,U_2^g ,\dots,U_k^g  \}$, and so there exits $b\in\Sy_k$ such that $U_{j}^g=W_{j^b}$ for all $j\in\{1,\dots,k\}$.
Recall that $G_v=\Sy_m\wr \Sy_k$.
Let $a$ be the image of $g$ under the homomorphism $G_v\to G_v/\Sy_m^k = \Sy_k$.
Then $V_i^g=V_{i^a}$ for all $i\in\{1,\dots,k\}$, and so
\[
(V_i\cap U_j)^g=V_i^g \cap U_j^g  =V_i^g \cap W_{j^b} =V_{i^a} \cap W_{j^b}\text{ for all } i,j\in\{1,\dots,k\}.
\]
This implies $|V_i\cap U_j|=|V_{i^a} \cap W_{j^b}|$ for all $i,j\in\{1,\dots,k\}$.
Hence $B=PAQ$, where $P$ and $Q$ are the permutation matrices corresponding to $a$ and $b$, respectively.

Conversely, suppose that $B=PAQ$ for some permutation matrices $P$ and $Q$.
Let $a,b\in\Sy_k$ be the permutations corresponding to $P$ and $Q$, respectively.
From $B=PAQ$ we deduce that $|V_i\cap U_j |=|V_{i^{a}} \cap W_{j^{b}}|$ for all $i,j\in\{1,\dots,k\}$.
Let $g\in\Sy_n$ such that $(V_i\cap U_j)^g=V_{i^{a}} \cap W_{j^{b}}$ for all $i,j\in\{1,\dots,k\}$.
Then
\begin{align*}
V_i^g=\bigcup_{j=1}^{k}(V_i\cap U_j)^g&=\bigcup_{j=1}^{k}(V_{i^{a}} \cap W_{j^{b}})=V_{i^{a}}, \\
U_j^g=\bigcup_{i=1}^{k}(V_i\cap U_j)^g&=\bigcup_{i=1}^{k}(V_{i^{a}} \cap W_{j^{b}}) =W_{j^{b}}.
\end{align*}
This implies that $g\in {G_v}$ and $u^g=\{U_1^g,\dots,U_k^g \}=\{W_{1^{b}},\dots,W_{k^{b}} \}=w$.

Next we prove (b).
Suppose that there exists some $g\in G$ interchanging $v$ and $w$, that is, $w^g=v$ and $v^g=w$.
Then there exist $a,b\in \Sy_k$ such that
\[
W_i^g=V_{i^a} \text{ and } V_j^g=W_{j^b}\text{ for all } i,j\in\{1,\dots,k\}.
\]
It follows that
\[
|V_j  \cap W_i| =|V_j^g  \cap W_i^g|=|W_{j^b}\cap V_{i^a}|=|V_{i^a}\cap W_{j^b}|.
\]
Hence $A^{\mathsf{T}}=PAQ$, where $P_1$ and $P_2$ are the permutation matrices corresponding to $a$ and $b$, respectively.

Suppose conversely that $A^{\mathsf{T}}=PAQ$ for some permutation matrices $P$ and $Q$.
Let $a,b\in\Sy_k$ be the permutations corresponding to $P$ and $Q$, respectively.
From $A^{\mathsf{T}}=PAQ$ we derive that $|V_j\cap W_i|=|V_{i^a}\cap W_{j^b}|$ for all $i,j\in\{1,\dots,k\}$.
Let $g\in\Sy_n$ such that $(V_j\cap W_i)^g=V_{i^a}\cap W_{j^b}$ for all $i,j\in\{1,\dots,k\}$.
Then
\begin{align*}
V_j^g=\bigcup_{i=1}^k(V_j\cap W_i)^g&=\bigcup_{i=1}^k(V_{i^a}\cap W_{j^b})=W_{j^{b}},\\
W_i^g=\bigcup_{j=1}^k(V_j\cap W_i)^g&=\bigcup_{j=1}^k(V_{i^a}\cap W_{j^b})=V_{i^a}.
\end{align*}
This implies that $v^g=w$ and $w^g=v$.

Now we prove (c).
Let $g$ be an element of $G\setminus T$ interchanging $v$ and $w$.
If $G_{vw}\nsubseteq T$, then taking some $x\in G_{vw}\setminus T$ we have $xg \in T$ and $(v,w)^{xg}=(v,w)^g=(w,v)$.
Suppose that $(v,w)^{t}=(w,v)$ for some $t\in T$.
Since $v^t=w=v^g$ and $w^t=v=w^g$, it follows that $t\in G_vg$ and $t\in G_wg$.
Hence $t\in(G_vg)\cap(G_wg)=(G_v\cap G_w)g=G_{vw}g$.
As $t\in T$ and $g\notin T$, this implies that $G_{vw}\nsubseteq T$.

Finally, suppose that $v$ and $w$ are interchanged by some element in $G\setminus T$ but no element in $T$.
Then by part~(c), we have $\G_{vw}\leq T$.
Notice that for fixed $i,j \in \{1,\dots,k\}$, the group $\Sym(V_i \cap W_j)$ stabilizes  set  $V_{i'} \cap W_{j'}$ set-wise for all $i',j'\in \{1,\dots,k\}$ and hence  $\Sym(V_i \cap W_j)\leq G_{vw}$.
If $|V_i \cap W_j|\geq 2$, then $\Sym(V_i \cap W_j)$ contains an odd permutation, contradicting $\G_{vw}\leq T$. Therefore $|V_i \cap W_j|\leq 1$ for all $i,j\in \{1,\dots,k\}$. In particular, we have
$m=|V_1|=\sum_{j=1}^{k}|V_1 \cap W_j|\leq k$.
\end{proof}

The second class is the group $\Omega_{2m+1}(q)$, where $m\geq 2$ and $q$ is odd,   acting on the set of non-singular $1$-spaces. We remark that $\Omega_{2m+1}(q)$ has two orbits on this set, with point stabilizer $\Omega^\epsilon_{2m}(q).2$ for $\epsilon=+$ or $-$, respectively. In Liebeck, Praeger and Saxl~\cite[Proposition~1]{LPS1988}, it was proved that all $\Omega_{2m+1}(q)$-suborbits are self-paired when $m=3$ and $\epsilon=-$.
As mentioned in the remark after~\cite[Proposition~1]{LPS1988},  the methods therein can be used to obtain precise information on the action of $\Omega_{2m+1}(q)$, for arbitrary $m$ and $\epsilon$. This leads to the following lemma.

\begin{lemma}\label{lm:Omeganonsingular}
Let $V$ be a vector space of dimension $2m+1$ over a field $\mathbb{F}_q$, where $q$ is odd and $m\geq 2$, let $T=\Omega(V)$ and let $\Delta$ be a $T$-orbit on the set of  non-singular $1$-spaces in $V$. Then all $T$-suborbits on $\Delta$ are self-paired.
\end{lemma}

The third class is  $\Sp_{2m}(q)$ for even $q$ with point  stabilizer $\GO_{2m}^{\pm}(q)$,  which is due to Inglis{~\cite{Inglis1990}}.

\begin{lemma}[{\cite[Theorem 1]{Inglis1990}}] \label{lm:SpOmq}
Let $V$ be a vector space of dimension $2m$ over a field $\mathbb{F}_q$, where $q$ is even, let $T=\Sp(V)$, and let $\Delta_1$ and $\Delta_2$ be the cosets of $\SO^{+}(V)$ and $\SO^{-}(V)$ in $T$, respectively. Then all $T$-suborbits on $\Delta_1$ or $\Delta_2$ are self-paired.
\end{lemma}

\subsection{Parabolic subgroups of finite simple groups of Lie type}\label{subsec:Parobolic}

We recall some results on parabolic subgroups of finite simple groups of Lie type.
The reader may be referred to~\cite{Carter1972} and~\cite[Chapter 2]{Carter1985}, or~\cite[Chapter 10]{BCN-book} for details.
Let $T={}^d\Sigma(q)$ be a simple group of Lie type over a field of order $q^a$, where $d\in \{1,2,3\}$ ($T$ is untwisted if $d=1$, and twisted otherwise), and $\Sigma$ is a Lie notation, and $q=p^f$ with $p$ prime.
Write $T=\Sigma(q)$ for short if $d=1$.
% over a finite field $\mathbb{F}_q$, where $q=p^f$ with $p$ prime.
There is a root system $\Phi=\Phi^{+}\cup \Phi^{-}$ such that  $T=\langle X_r \mid r \in \Phi  \rangle$, where $\Phi^{+}$ is the set of positive roots, $\Phi^{-}$ is the set of negative roots, and $ X_r$ is the root subgroup corresponding  to the root $r$.
Let $\Pi \subseteq \Phi^{+}$ be  a set of fundamental roots.
For  $r \in \Pi$, let $w_r$ be the   fundamental reflection  corresponding  to   $r$.
All the fundamental reflections generate the \emph{Weyl group} of $T$, denoted by $W$.
Each element $z\in W$ is a product of fundamental reflections, and the \emph{length} of $z$, denoted by $\ell(z)$, is the number of smallest integer $m$ such that $z$ is a product of $m$ fundamental reflections.

By~\cite[Proposition 8.2.1 and~Theorem 13.5.4]{Carter1972},
$T$ has a $(B,N)$-pair, that is, $G$ has a pair of subgroups $(B,N)$ such that the following hold:
\begin{enumerate}[\rm (a)]
\item $T=\langle B,N\rangle$;
\item $H:=B\cap N \unlhd N$;
\item $N/H \cong W =\langle w_r\mid r \in  \Pi\rangle$;
\item if $r \in \Pi$ and $z \in W$, then $(Bw_rB)(BzB)\subseteq Bw_rzB \cup BzB$  and $w_rBw_r\neq B$.
\end{enumerate}
Furthermore, $T=BNB=BWB$, $H$ normalizes each root subgroup $X_r$, and $X_r^{z}=z^{-1}X_r z=X_{r^{z^{-1}}}$ for all $z \in W$ and $r \in \Phi$.
The group $B$ is called a \emph{Borel} subgroup, and it has structure $B=U{:}H$, where  $U:=\langle X_r\mid r\in \Phi^+\rangle$.
A subgroup $P $ of $T$ is said to be \emph{parabolic} if $B^g\leq P$ for some $g\in T$.

For  $J \subseteq \Pi$, let $W_J= \langle w_{r} \mid  r \in J\rangle$.
The set $P_J:=BW_JB$ forms a subgroup of $T$ (see ~\cite[Proposition 8.2.2]{Carter1972}).
Every parabolic subgroup of $T$ is conjugate to $P_J$ for some  $J \subseteq \Pi$ (see~\cite[Theorem 8.3.2]{Carter1972}).
For $I,J\subseteq \Pi$, a double coset  $W_I z W_J$ has a unique element $z'$ with shortest length (see~\cite[Proposition 2.7.3]{Carter1985}).
Let $D_{I,J}$ be the set of such elements $z'$ with $z$ running over $W$.
Then $D_{I,J}$ is a set of representations of $(P_I,P_J)$-double cosets in $T$ (see~\cite[Proposition 2.8.1]{Carter1985}).
In particular, $D_{J,\emptyset}$ is a set of representations of right $ P_J $-cosets in $G$.
For $z \in D_{J,J}$, note that by the definition of $D_{J,J}$ that $z^{-1} \in P_J z P_J$ if and only if $z$ is an involution.
Thus we have  the following lemma.
\begin{lemma}\label{lm:PJsuborbits}
Let $T$ be a simple group of Lie type, and let $\Pi$ be a set of fundamental roots of $T$.
Suppose that $T$ acts transitively on a set $\Delta$ with stabilizer a parabolic subgroup.
Then there exists a subset $J \subseteq \Pi$ and a point $v\in \Delta$ such that $P_J=T_v$.
Moreover, for every non-self-paired $T$-suborbit $\Lambda$ relative to $v$, there is a non-involution element $z\in D_{J,J}$ such that $v^z \in \Lambda$.
\end{lemma}

The parabolic subgroup $P_J$ is said to be \emph{maximal} if $|J|=|\Pi|-1$.
The following Lemma~\ref{lm:PJmaximal} is a consequence of~\cite[Proposition~10.4.7 and~Theorem~10.4.11]{BCN-book}, which shows that most $T$-suborbits relative to $v$ are self-paired when $T_v$ is a maximal parabolic subgroup.

\begin{lemma}[Brouwer-Cohen-Neumaier] \label{lm:PJmaximal}
Let $T$ be a simple group of Lie type acting on a set $\Delta$ with stabilizer a maximal parabolic subgroup, and let $v  \in \Delta$. If there is a non-self-paired $T$-suborbit relative to $v$, then one of the following holds:
\begin{enumerate}[\rm (a)]
\item $T=\POm_{2n}^{+}(q)$ with $n\geq 5$, and $T_v$ is of type $\A_i\D_{n-1-i}$ with $i \in \{\lfloor{n/2}\rfloor,\ldots,n-3 \}$, so that $T_v$ is the stabilizer of a totally singular $k$-dimensional subspace for some $k \in \{\lfloor{n/2}\rfloor+1,\ldots,n-2 \}$;
\item $T=\F_4(q)$, and $T_v$ is of type $\A_1\A_2$;
\item $T={}^2\E_6(q)$, and $T_v=[q^{29}].(\PSL_3(q^2)\times\SL_2(q)){:}(q-1))$ or
$[q^{31}].(\SL_3(q)\times\SL_2(q^2)){:}(q^2-1)/(3,q+1)$;
\item  $T=\E_6(q)$, and $T_v$ is of type $\A_1\A_4$ or $\A_1 \A_2\A_2$;
\item  $T=\E_7(q)$, and $T_v$ is of type $\A_1\A_5$, $\A_1\A_2\A_3$, $\A_2\A_4$ or $\A_1\D_5$;
\item  $T=\E_8(q)$, and $T_v$ is of type $\A_7$, $\A_1\A_6$, $\A_1\A_2\A_4$, $\A_3\A_4$, $\A_2\D_5$ or $\A_1\E_6$.
\end{enumerate}
\end{lemma}

\begin{proof}
Let $T_v=P_J$ be a maximal parabolic subgroup of $T$, where $|J|=|\Pi|-1$. By \cite[Proposition~10.4.7]{BCN-book}, every element of $D_{J,J}$ has order at most $2$ if and only if the character of the permutation representation of $W$ on $W_J$ is multiplicity-free.
Since $D_{J,J}$ is a set of representations of $(P_J,P_J)$-double cosets in $T$, we conclude that all $T$-suborbits on the set of right $P_J$-cosets are self-paired if the character of the permutation representation of $W$ on $W_J$ is multiplicity-free.
Moreover,~\cite[Theorem~10.4.11]{BCN-book} asserts that this is the case if the Coxeter graph $Y_{n,j}$ corresponding to $(W,W_J)$ is one of:
\begin{enumerate}[{\rm(i)}]
\item $\A_{n,i}$, $\B_{n,i}$, $\C_{n,i}$,
\item $\D_{n,i}$ for $i\leq n/2$ or $i\in \{n-1,n\}$,
\item $\E_{6,1}$, $\E_{6,2}$, $\E_{6,6}$,  $\E_{7,1}$, $\E_{7,2}$, $\E_{7,7}$,  $\E_{8,1}$, $\E_{8,8}$, $\F_{4,1}$, $\F_{4,4}$, $\mathrm{I}_{2,1}^{m}$, $\mathrm{I}_{2,2}^{m}$.
\end{enumerate}
Here $Y_{n,j}$ is defined in~\cite[\textsection{10.2}]{BCN-book}; in particular, $Y_n$ is the diagram for $W$ in~\cite[Table~10.1]{BCN-book}, and $j$ is the node in $Y_n$ corresponding to the unique fundamental root in $\Pi\setminus J$.

We now discuss the possibility of $T$ and $T_v$ such that there is a non-self-paired $T$-suborbit relative to $v$.
The above paragraph shows that the Coxeter graph $Y_{n,j}$ corresponding to $T$ and $T_v$ is not in~(i)--(iii).
Note that the Weyl group of a twisted group is isomorphic to the Weyl group of some untwisted group, see~\cite[Table~10.7]{BCN-book} for example.

First assume that $T$ is a classical group. Then either $T$ is untwisted with Weyl group $\A_n$, $\B_n$, $\C_n$ or $\D_n$, or $T$ is $\PSU_{n}(q)={}^2\A_{n-1}(q)$ or $\POm^{-}_{2n}(q)={}^2\D_n(q)$ with Weyl group $\B_{\lceil(n-1)/2\rceil}$ or $\B_{n-1}$, respectively. Since $Y_{n,j}$ is not in~(i)--(ii), we conclude that $T=\POm^{+}_{2n}(q)$ with $T_v$ of type $\A_i\D_{n-i-1}$ for some integer $i \in \{\lfloor{n/2}\rfloor,\ldots,n-3 \}$, and so $T_v$ is the stabilizer of a totally singular $(i+1)$-dimensional subspace, as in part~(a).

Next assume that $T$ is a group of exceptional Lie type. Since $Y_{n,j}$ is not in~(iii), the Coxeter graph $Y_{n,j}$ is one of:
\[
\E_{6,3},\ \E_{6,4},\ \E_{6,5}, \ \E_{7,3},\ \E_{7,4},\ \E_{7,5},\ \E_{7,6},\ \E_{8,2},\ \E_{8,3},\ \E_{8,4},\ \E_{8,5},\ \E_{8,6},\ \E_{8,7},\ \F_{4,2},\ \F_{4,3}.
\]
For $\F_{4,2}$ and $\F_{4,3}$ we obtain~(b) and~(c) of the lemma (for~(c), we refer to~\cite[Table~23.2]{GD-book} for the type of $T_v$ and then \cite[Theorem~1.1]{Craven-arxiv} for the precise group structure\footnote{Some typos from~\cite{Craven-arxiv} have been corrected here according to our personal communication with the author.} of $T_v$).
For the other candidates of $Y_{n,j}$, we obtain~(d)--(f).
\end{proof}

The Weyl group of a twisted group ${}^d\Sigma(q)$ is a subgroup of the Weyl group of the corresponding untwisted group $\Sigma(q)$ (see~\cite[p.217]{Carter1972}).
For $x \in W$, let $|x|=\ell(x)$ if $d=1$, and let $|x| $ be the length of $x$ in the Weyl group of $\Sigma(q)$.
% Brouwer and  Cohen~\cite[Proposition 1]{BC1983} provide a formula to compute the number of right $ P_J $-cosets in $P_JzP_J$ with $z\in D_{J,J}$.
%The following lemma is from~\cite[Proposition~10.7.3]{BCN-book}.

\begin{lemma}[{\cite[Proposition~10.7.3]{BCN-book}}]\label{lm:PJlength}
Let $T={}^d\Sigma(q)$ be a simple group of Lie type, and let  $\Pi$ be a set of fundamental roots of $T$.  For $J\subseteq \Pi$ and $z \in D_{J,J}$,
\[
\frac{|P_J|}{|P_J\cap P_J^z|}=\sum_{x \in D_{J,\emptyset} \cap zW_J} q^{|x|}.
\]
\end{lemma}

%To find and compute the lengths of non-self-paired suborbits of $G$ acting on $[G:P_J]$, we can do  computation on the Wyle group in {\sc Magma}~\cite{Magma} with commands {\bf CoxeterGroup}, {\bf Transversal}, {\bf StandardParabolicSubgroup}, {\bf TransversalElt} and so on.

%For example, when $G=\PSL_5(2)$ with $P_J=[2^7]{:}\PSL_3(2)$, the stabilizer in $G$ of a $(1,4)$-flag, we can do as follows
%\begin{verbatim}
%W := CoxeterGroup(GrpFPCox,"A4");J:={2,3};
%DJJ,_:=Transversal(W, J, J); //D_{J,J}
%W0,phi := CoxeterGroup(GrpPermCox, W);
%DJJ1:=[phi(DJJ[i]): i in [2..#DJJ]|  Order(DJJ[i]) ne 2];
%print "number of non-self-paired suborbits:", #DJJ1;
%DJ0:=Transversal(W, J);  //D_{J,\emptyset}
%WJ:= StandardParabolicSubgroup(W0, J);
%for w in DJJ1 do
%DJ0_wWJ:=[g: g in DJ0|w eq TransversalElt(W0,WJ,phi(g),WJ)];
%DJ0_wWJ1:=[Length(w): w in DJ0_wWJ];
%DJ0_wWJ1; //the set [l(v):v  \in D_{J,\emptyset}  \cap wW_J]
%end for;
%\end{verbatim}

By~\cite[Proposition 8.5.1]{Carter1972} we have
\[
P_J=\langle H,X_r \mid r \in \Phi^+ \cup \Phi_J\rangle,
\]
where $\Phi_J$ is the set of roots spanned by fundamental roots in $J$.
For a double coset $P_JzP_J$, where $z \in D_{J,J}$,~\cite[Theorem  2.8.7]{Carter1985}  implies
\[
P_J \cap P_J^z= \langle H,X_r \mid r \in (\Phi^+ \cup \Phi_J) \cap (\Phi^+ \cup \Phi_J)^{z^{-1}} \rangle,
\]
and~\cite[Proposition  2.5.9]{Carter1985} implies
%\[
%B \cap B^z \cap B^{z^{-1}}=\langle H,X_r \mid r \in  \Phi^+   \cap (\Phi^+ )^{z^{-1}} \cap (\Phi^+ )^{z} \rangle.
%\]
\begin{equation}\label{EqBBB}
B \cap B^z \cap B^{z^{-1}}=\langle H,X_r \mid r \in  \Phi^+   \cap (\Phi^+ )^{z^{-1}} \cap (\Phi^+ )^{z} \rangle.
\end{equation}
Thus we have the following lemma.

\begin{lemma} \label{lm:PJroots}
Let $T={}^d\Sigma(q)$ be a untwisted simple group of Lie type, and let $\Phi$ be a root system for $T$ with  $\Pi$  a set of fundamental roots.  For $J\subseteq \Pi$ and $z \in D_{J,J}$, let $m_1$, $m_2$, $m_3$ be the numbers of positive roots in $\Phi^+ \cup \Phi_J$, $(\Phi^+ \cup \Phi_J) \cap (\Phi^+ \cup \Phi_J)^{z^{-1}}$,  $\Phi^+   \cap (\Phi^+ )^{z^{-1}} \cap (\Phi^+ )^{z}$, respectively.
Then
\[
|P_J|_p=q^{m_1},\ \ |P_J \cap P_J^z|_p=q^{m_2},\ \ |P_J \cap P_J^z \cap P_J^{z^{-1}}|_p \geq q^{m_3}.
\]
\end{lemma}

In {\sc Magma}~\cite{Magma}, commands {\sf CoxeterGroup}, {\sf Transversal} and {\sf StandardParabolicSubgroup} are used to compute  $W$, $D_{I,J}$ and $W_J$, respectively.

%  The next lemma is consequence of direct computation.
%
%  \begin{lemma}\label{lm:smallT}
%Let $T$ be a simple  group of Lie type. Then, up to isomorphism, $|T|^{2/3} \leq 30758154560$ if and only if $T$ is one of the following groups:
%\begin{itemize}
%\item $\PSL_n(q)$ or $\PSU_n(q)$ with $(n,q)=(2, q\leq 170003)$, $(3,q\leq 97)$,   $(4,q\leq 11)$, $(5,q\leq 4)$,  $(6, 2)$ or $(7, 2)$;
%
%\item $\PSp_n(q)$ with $(n,q)=(4,q\leq 37)$,  $(6,q\leq 5)$,  or $(8, 2)$
%
%\item $\POm_n^{\epsilon}(q)$  with $(n,q)= (7,  3)$, $(7,  5)$, $(8,  2)$, $(8,  3)$ or  $(10,  2)$;
%\item  $G_2(q)$ with $q\leq 13$,  $ {}^2B_2(q)$ with $q\in\{2^3,2^5,2^7,2^9,2^11\}$, $ {}^2G_2(3^3)$ with $q\in\{3^3,3^5 \}$,   $ {}^3D_4(q)$ with $q\leq 4$, $F_4(2)$,  or  ${}^2F_4(2)'$.
%
%\end{itemize}
%\end{lemma}

\section{Strategy of the proof}
In this section, we  explain the basic ideas to prove Theorem~\ref{th:main}.

% Let $\Gamma$ be a  vertex-primitive $2$-arc-transitive digraph of valency at least $2$ with fewest vertices.

\subsection{Reduction to almost simple groups}
\ \vspace{1mm}

Let $\Gamma$ be a  vertex-primitive $2$-arc-transitive digraph of valency at least $2$ with fewest vertices.
Then $|V(\Ga)|\leq |V(\Ga_7)|=30758154560$.
In~\cite[Section 5]{P1997}, primitive groups were divided into eight types.  If $\Aut(\Ga)$ is not of type $\mathrm{AS}$,   then the digraph $\Ga$ is well characterized in~\cite{P1989} and \cite{GX2017}.
This allows us to reduce the proof of Theorem~\ref{th:main} to the  $\mathrm{AS}$ type. More precisely, Praeger~\cite[Theorem 3.1]{P1989} states that if $\Aut(\Ga)$  has a regular normal subgroup, then $\Ga$ is a directed cycle.
This implies that  $G$ is not of  type $\mathrm{HA}$, $\mathrm{HC}$, $\mathrm{HS}$ or $\mathrm{TW}$.
If  $\Aut(\Ga)$  is of type $\mathrm{SD}$, then~\cite[Theorem~1.2 and~Construction~3.1]{GX2017} imples $|V(\Ga)|=|S|^{|S|-1}$ for some nonabelian simple group $S$, which would lead to $|V(\Ga)|>60^{59}>30758154560$,  a contradiction.
If $\Aut(\Ga)$  is of type $\mathrm{CD}$ or $\mathrm{PA}$, then from~\cite[Corollary 1.4 and 1.5]{GX2017} we see that $\Ga$ is a direct product of $m>1$ copies of some  vertex-primitive $2$-arc-transitive digraph, contradicting that $\Ga$ has the smallest order.
Therefore, $\Aut(\Ga)$  is of type $\mathrm{AS}$.

Based on the above analysis, in the rest of the paper, we make the following hypothesis.

\begin{hypothesis}\label{hy:1}
 Let $\Gamma$ be a connected $G$-vertex-primitive $(G, 2)$-arc-transitive digraph of valency at least $2$, where $G$
is almost simple with socle $T$.
Take an arc $u \rightarrow v$ of $\Gamma$.
Let $g$ be an element of $G$ such that $v=u^g$, and let $w = v^g$.
Then $(u ,v,w)$ is a $2$-arc in $\Ga$.
\end{hypothesis}

In Sections~\ref{Sec1} and~\ref{Sec2}, we will apply the Classification of Finite Simple Groups to consider candidates for $T$ one by one, thus finishing the proof of Theorem~\ref{th:main}.
Our argument is supplemented by computer computation in  {\sc Magma}~\cite{Magma}.
Before the separate treatment on each simple group $T$, we make some discussion in the general case in the rest of this section.

\subsection{Two lemmas}
\ \vspace{1mm}

Since $G$ is vertex-primitive, $G_v$ is a maximal subgroup of $G$.
By Lemma \ref{pro:valency}, $\Ga$ has valency at least $3$, that is, $|G_v|/|G_{uv}|\geq 3$.
By Lemma \ref{pro:Gvfactorization}, there holds $G_v=G_{uv}G_{vw}$.
In particular, $G_{uv}$ is not conjugate to $G_{vw}$ in $G_v$ by Lemma \ref{lm:basicfacs}.
Since $G_{uv}^g=G_{u^gv^g}=G_{vw}$, the factorization $G_v=G_{uv}G_{vw}$ is homogeneous.
The following lemma studies the action of $T$ on $\Gamma$.

\begin{lemma}\label{lm:T2at?}
Suppose that Hypothesis~$\ref{hy:1}$ holds, and let $t$ be the number of orbits of $T_{uv}$ on $\Ga^{+}(v)$.
Then the following hold:
\begin{enumerate}[\rm (a)]
\item $\Ga$ is $T$-arc-transitive;
\item $|G|/|T|=|G_{v}|/|T_{v}|=|G_{uv}|/|T_{uv}|=t|G_{uvw}|/|T_{uvw}|$;
\item $t=|T_v|/|T_{uv}T_{vw}|$;
\item if $T_v=T_{uv}T_{vw}$, then $\Ga$ is $(T,2)$-arc-transitive;
\item if $T_{uv}$ is conjugate to $T_{vw}$ in $T_v$, then $|T_v|$ divides $|G|/|T|$.
\end{enumerate}
\end{lemma}

\begin{proof}
Part~(a) follows from Lemma \ref{pro:Tarctrans}.

Note that  both $G$ and $T$ are transitive on $V(\Ga)$, and both $G_v$ and $T_v$ are transitive on $\Ga^{+}(v)$.
By Frattini's argument we obtain $G=TG_v$ and $G_v=G_{vw}T_{v}$, respectively, and so
\[
|G_{vw}/T_{vw}|=|G_{vw}/(G_{vw} \cap T_v)|=|T_vG_{vw}/T_v|=|G_v/T_v|=|G_vT/T|=|G/T|.
\]
Since $\Ga$ is $(G,2)$-arc-transitive, $G_{uv}$ is transitive on $\Ga^{+}(v)$.
Then as $T$ is normal in $G$, we derive that the orbits of $T_{uv}$ on $\Ga^{+}(v)$ have the same lengths.
Hence
\[
t \frac{|T_{uv}|}{|T_{uvw}| } = \frac{|G_{uv}|}{|G_{uvw}| },
\]
and so $t|G_{uvw}|/|T_{uvw}|=|G_{vw}|/|T_{vw}|$.
Thus part~(b) holds.
Moreover, the above equality yields $t|T_{uv}|/|T_{uvw}|=|\Ga^{+}(v)|=|T_v|/|T_{vw}|$ as $T$ is arc-transitive.
It follows that
\[
t=\frac{|T_v||T_{uvw}|}{|T_{uv}||T_{vw}|}=\frac{|T_v||T_{uv}\cap T_{vw}|}{|T_{uv}||T_{vw}|}=\frac{|T_v|}{|T_{uv}T_{vw}|},
\]
proving part~(c).
Part~(d) is a consequence of part~(c) and Lemma~\ref{pro:Gvfactorization}.

Suppose that $T_{vw}=T_{uv}^h$ for some $h\in T_v$.
Take $N=T_{vw}$.
Then $N=T_{uv}^h=T_{u^hv^h}=T_{u^hv}$.
Note that $u^h \in \Ga^-(v^h)=\Ga^-(v)$.
% Consider the action of $N$ on $\Ga^{+}(v)$.
% Since $N=T_{u^hv}\unlhd G_{u^hv}$ and $G_{u^hv}$ is transitive on $\Ga^{+}(v)$, orbits of $N$ on $\Ga^{+}(v)$ have the same length, and hence $N=T_{vw}$ fixes each vertex in $\Ga^{+}(v)$ as $N$ fixes $w\in \Ga^{+}(v)$.
Consider the action of $N$ on $\Ga^{-}(v)$.
Since $N=T_{vw}\unlhd G_{vw}$, orbits of $N$ on $\Ga^{-}(v)$ have the same length, and so $N$ fixes each vertex in $\Ga^{-}(v)$ as $N$ fixes $u^h \in \Ga^{-}(v)$.
Therefore, for each arc $v'w'$ of $\Ga$, since $T$ is arc-transitive, $T_{v'w'}$ fixes $\Ga^{-}(v')$ pointwise.
Then by the connectivity of $\Ga$ (note that $\Ga$ is connected if and only if it is strongly connected, see~\cite[Lemma~2.6.1]{GR2001}), it follows that $N$ fixes each vertex of $\Ga$, whence $T_{vw}=N=1$.
Again since $T$ is arc-transitive, we obtain $T_{uv}=1$.
Thus part~(c) gives $|T_v|=t$, and then part~(b) implies that $|T_v|=t$ divides $|G|/|T|$, proving part~(e).
\end{proof}

%By Lemma \ref{pro:Tarctrans}, $\Ga$ is $T$-arc-transitive, which implies $\Ga$ is obtained from a non-selfpaired orbital of $T$ acting on $[T:(T_{v})]$.
%

Recall that $\Rad(G_v)$ denotes the largest solvable normal subgroup of $G_v$.
Throughout this paper, we set
\[
\overline{G_v}=G_v/\Rad(G_v),\ \overline{G_{uv}}=G_{uv} \Rad(G_{ v})/\Rad(G_{ v}),\ \overline{G_{vw}}=G_{vw}\Rad(G_{ v})/\Rad(G_{ v}).
\]
From the factorization $G_v=G_{uv}G_{vw}$ we deduce that
\begin{equation}\label{EqnFac}
\overline{G_v}=\overline{G_{uv}}\,\overline{G_{vw}}.
\end{equation}
Note that $\overline{G_{uv}}\cong G_{uv}/(\Rad(G_{v}) \cap G_{uv})$  and $\overline{G_{vw}}\cong G_{vw}/(\Rad(G_{v}) \cap G_{vw})$.
Since $G_{uv}\cong G_{vw}$, it follows that $ \overline{G_{uv}}$ and $\overline{G_{vw}}$ have the same nonsolvable composition factors (counting multiplicities).
Note that, if $G_{v}$ has a unique nonsolvable   composition factor, then $\overline{G_v}$ is almost simple.

A group $X$ is said to be \emph{quasisimple} if $X=X'$ and $X/\Z(X)$ is nonabelian simple.
If $Y$ is a subgroup of $X$ with a composition factor $X/\Z(X)$, then $Y\Z(X)/\Z(X)=X/\Z(X)$ and so $Y'=(Y\Z(X))'=X'=X$, which leads to $Y=X$.
In other words, $X$ is the unique subgroup of $X$ with a composition factor $X/\Z(X)$.

\begin{lemma}\label{lm:Hsiquasi}
Let $\mathcal{T}_1$ and $\mathcal{T}_2$ be as in~\eqref{eq:T1T2}.
Suppose that Hypothesis~$\ref{hy:1}$ holds and $G_{v}$ has a unique nonsolvable composition factor $M$.
\begin{enumerate}[\rm (a)]
\item If $G_{v}$ is almost simple, then $M \in \mathcal{T}_1$, both $G_{uv}$ and $G_{vw}$ are core-free in $G_v$, the factorization $ G_v=G_{uv}G_{vw}$ satisfies Lemma~$\ref{lm:ASfac}$, and $\Ga$ is $(T,2)$-arc-transitive.
\item If $G_{v}^{(\infty)}$ is quasisimple, then $M\in\mathcal{T}_1\cup\mathcal{T}_2$, both $\overline{G_{uv}}$ and $\overline{G_{vw}}$ are core-free in $\overline{G_{v}}$, and the factorization $\overline{G_{v}}=\overline{G_{uv}}\,\overline{G_{vw}}$ satisfies either Lemma~$\ref{lm:ASfac}$ or Lemma~$\ref{lm:ASfac2}$.
\item If $M \notin\mathcal{T}_1 \cup \mathcal{T}_2$, then both $G_{uv}$ and $G_{vw}$ have a composition factor isomorphic to $M$.
\end{enumerate}
\end{lemma}
\begin{proof}
Part~(a) is obtained directly from~\cite[Corollary 3.4]{GLX2019} and Lemma~\ref{lm:T2at?}(d).
 % (a) Suppose that $G_{v}$ is almost simple. Then $\Soc(G_{v})\cong M$.  By~\cite[Corollary 3.4]{GLX2019}, $S$ is not contained in $G_{uv}$ or in $G_{vw}$, and  $(G_{v}, G_{uv}, G_{vw})$ satisfies (a)-(d) of Lemma \ref{lm:ASfac}. Note that now $T_v=T_{uv}T_{vw}$. Hence  $\Ga$ is $(T,2)$-arc-transitive by Lemma \ref{pro:Gvfactorization}.

Assume that $G_v^{(\infty)}$ is quasisimple.
Then $G_{v}^{(\infty)}/\Z(G_{v}^{(\infty)})\cong M$ and $\overline{G_{v}}$ is an almost simple group with socle $M$.
Suppose for a contradiction that $\overline{G_{uv}}$ is not core-free in $\overline{G_{v}}$.
This means that $M$ is a normal subgroup of $\overline{G_{uv}}$, which together with $G_{uv}\cong G_{vw}$ implies that both $G_{uv}$ and $G_{vw}$ have a composition factor isomorphic to $M$.
Since $G_{uv}/G_{uv}^{(\infty)}$ and $G_{vw}/G_{vw}^{(\infty)}$ are solvable, it follows that both $G_{uv}^{(\infty)}$ and $G_{vw}^{(\infty)}$ have a composition factor isomorphic to $M$.
% Since
% \[
% G_{uv}/(G_{uv} \cap G_{v}^{(\infty)})\cong G_{uv}G_{v}^{(\infty)}/G_{v}^{(\infty)}\leq G_v/G_v^{(\infty)}
% \]
% is solvable, we derive that $G_{uv}^{(\infty)} \leq G_{uv} \cap G_{v}^{(\infty)} \leq G_{v}^{(\infty)}$.
Note that $G_{v}^{(\infty)}$ has no proper subgroup with a composition factor isomorphic to $M$ as $G_{v}^{(\infty)}$ is quasisimple.
We conclude that $G_{uv}^{(\infty)}=G_{v}^{(\infty)}=G_{vw}^{(\infty)}$.
Write $Y= (G_{uv}^{(\infty)})^g$.
Then $Y\leq G_{uv}^g=G_{vw}$, and hence
\[
Y/(Y\cap G_{vw}^{(\infty)})\cong Y G_{vw}^{(\infty)}/ G_{vw}^{(\infty)} \leq G_{vw}/ G_{vw}^{(\infty)}
\]
is solvable.
Thus $Y\cap G_{vw}^{(\infty)}$ has a composition factor isomorphic to $M$.
This combined with the fact that $G_{vw}^{(\infty)}=G_{v}^{(\infty)}$ is quasisimple yields $Y\cap G_{vw}^{(\infty)}=G_{vw}^{(\infty)}$.
Therefore, $Y=G_{vw}^{(\infty)}$, that is, $(G_{uv}^{(\infty)})^g= G_{vw}^{(\infty)}=G_{v}^{(\infty)} $, contradicting Lemma~\ref{pro:Gvnormal}.
As a consequence, both $\overline{G_{uv}}$ and $\overline{G_{vw}}$ are core-free in $\overline{G_{v}}$.
Then by Lemmas~\ref{lm:ASfac} and~\ref{lm:ASfac2},  part~(b) holds.

To prove part (c),  noticing $G_{uv} \cong G_{vw}$, we only need to prove that $G_{uv}$ has a composition factor isomorphic to $M$.
Suppose for a contradiction that this is not the case.
Since $\overline{G_{v}}$ is an almost simple group with socle $M$, this implies that $\overline{G_{uv}}$ is core-free in  $\overline{G_v}$, and so is $\overline{G_{vw}}$ as $G_{uv} \cong G_{vw}$.
Recall from~\eqref{EqnFac} that $\overline{G_v}=\overline{G_{uv}}\,\overline{G_{vw}}$.
Then it follows from Lemmas~\ref{lm:ASfac} and~\ref{lm:ASfac2} that $M \in\mathcal{T}_1 \cup \mathcal{T}_2$, a contradiction.
\end{proof}

\subsection{Computational methods}
\ \vspace{1mm}

Some of the results in the sequel will be obtained by computation in {\sc Magma}~\cite{Magma}.
Recall in Hypothesis~\ref{hy:1} that $G$ is an almost simple group with socle $T$.
% In general, we use the {\sc Magma} command {\bf AutomorphismGroupSimpleGroup} to construct the automorphism group $\Aut(T)$ of $T$.
% All candidates for $G$ are obtained  from $\Aut(T) $, and all candidates  for $G_v$ are obtained from $G$  by the {\sc Magma} command {\bf MaximalSubgroups}.
Since $G_v=G_{uv}G_{vw}$, it follows from Lemma~\ref{lm:basicfacs} that
\begin{equation}\label{EqnOrder}
|G_v||G_{uv} \cap G_{vw}|=|G_{uv}||G_{vw}|.
\end{equation}
In particular, $|G_v|$ divides $|G_{uv}|^2=|G_{vw}|^2$.
Let $|G_v|=p_1^{s_1}\cdots p_r^{s_r}$ for some distinct primes $p_i$ with $i\in\{1,\dots,r \}$.
Then
\begin{equation}\label{EqnDivisor}
|G_{vw}|=|G_{uv}|\text{ is divisible by }p_1^{ \lceil{ s_1/2 }\rceil }\cdots p_r^{ \lceil{s_r/2}\rceil }.
\end{equation}
%The {\sc Magma} command {\bf  Subgroups(H :OrderMultipleOf:=m)} can compute all subgroups of $H$ with order multiple of  $m$. By this command,
Thus $G_{uv}$ and $G_{vw}$ are such subgroups of $G_v$ satisfying~\eqref{EqnOrder},~\eqref{EqnDivisor} and that $G_{uv}$ is conjugate to $G_{vw}$ in $G$.

Sometimes, the above process to find the candidates for $G_{uv}$ and $G_{vw}$ can be optimized when $G_v$ has a certain structure.
For example, the Thompson group $G=\Th$ has some maximal subgroups isomorphic to $(3 \times \G_2(3)){:}2$, $2^5.\PSL_5(2)$ or $3.[3^8].2\Sy_4$.
We can rule  out the case  $ G_v=(3 \times\G_2(3)){:}2$  directly according to Lemma \ref{lm:Hsiquasi}(b), because  $ G_v^{(\infty)}$ is quasisimple but the unique nonsolvable composition factor $\G_2(3)$ of $G_v$ is not in $\mathcal{T}_1\cup\mathcal{T}_2$.
If $G_v=2^5.\PSL_5(2)$, then we see from Lemma \ref{lm:Hsiquasi}(c) that both $G_{uv}$ and $G_{vw}$ have a composition factor $\PSL_5(2)$, and hence $|G_{uv}|=|G_{vw}|$ is divisible by $|\PSL_5(2)|=2^{10}\cdot3^2\cdot5\cdot7\cdot31$.
Suppose that $G_v=3.[3^8].2\Sy_4$.
Let $N=3.[3^8]$ be the normal subgroup of $G_v$ such that $G_{v}/N \cong 2\Sy_4$.
From the factorization $G_v=G_{uv}G_{vw}$, we derive that  $G_{v}/N=(G_{uv}N/N)(G_{vw}N/N)$.
Since $G_{uv}\cong G_{vw}$ and $N$ is a $3$-group, it follows that the Sylow $2$-subgroup of $G_{uv}N/N$ is isomorphic to that of $G_{vw}N/N$.
Searching the factorizations of $G_{v}/N\cong 2.\Sy_4$ with the Sylow 2-subgroups of two factors isomorphic, we obtain that $G_{uv}N/N=G_{vw}N/N=G_{v}/N\cong 2\Sy_4$, and so $|G_{uv}|_2=|G_{vw}|_2$ is divisible by $|2\Sy_4|_2=2^4$.
This together with~\eqref{EqnDivisor} implies that $|G_{uv}|=|G_{vw}|$ is divisible by $2^4\cdot3^5$.

% When $T$ is a  simple  classical group with large order,  the command    {\bf MaximalSubgroups} may dose not work,  whence we use the  command {\bf ClassicalMaximals} to constructed the preimage of $T_v$ in a certain quasisimple group of $T$.

%By (a) of  Lemma \ref{lm:T2at?}, $|T_{vw}|^2$ is divisible by $|T_v|/|\Out(T)|$. Note that $T_{uv}$ is  not conjugate to $T_{vw}$ in $T_v$ if $|T_v|\geq |G/T|$, by (c) of Lemma \ref{lm:T2at?}, but they are conjugate in $T$ by the arc-transitivity of $T$. Then we need to check  whether $|T_v||T_{uv} \cap T_{vw}|/|T_{vw}|^2 $ divides $|\Out(T)|$.

Recall that $\Ga$ is $T$-arc-transitive by Lemma~\ref{lm:T2at?}(a).
Let $z \in T$ such that $(v,w)=(u,v)^z=(u^z,v^z)$.
From Lemma~\ref{lm:orbital} we see that $w$ lies in a non-self-paired $T$-suborbit  relative to $v$ and $z^{-1} \notin T_vzT_v$.
When $T={}^d\Sigma(q)$ is a simpe group of Lie type,  where $q=p^f$ with $p$ a prime, and $T_v $ is a parabolic subgroup of $T$, according to  Lemma~\ref{lm:PJsuborbits},  we may let $T_v=P_J$ for some $J \subseteq \Pi$, and  $z\in D_{J,J}$ such that $z$ is not an involution, where $\Pi$ and $D_{J,J}$ are defined as Subsection~\ref{subsec:Parobolic}.
It follows that $T_{vw}=P_J \cap P_J^z$, $T_{uv}=P_J \cap P_J^{z^{-1}}$ and $T_{uvw}=P_J \cap P_J^{z^{-1}}\cap P_J^z$.
Form Lemma~\ref{lm:T2at?} we see that $|T_{vw}|^2 =t |T_v| |T_{uvw}|$ for some $t$ dividing $|G/T|$.
In particular,
\begin{equation}\label{EqTvroot}
|T_{vw}|_p^2=t_p|T_v|_p|T_{uvw}|_p.
\end{equation}
This equation can be verified by computing the roots of $T$ (see Lemma~\ref{lm:PJroots}).

\section{Non-classical groups}\label{Sec1}

In this section, we deal with alternating groups, sporadic simple groups and exceptional groups of Lie Type for the candidates of $T$.  We suppose  Hypothesis~$\ref{hy:1}$ throughout this section.

\subsection{Alternating groups} \label{subsec:3.2}
\ \vspace{1mm}

In this subsection, let $T=\A_n$ with $ n\geq 5$.
By the classification of maximal subgroups of alternating and symmetric groups (see~\cite{LPS1987}),  one of the following holds:
\begin{enumerate} [\rm (i)]
\item $G_{v}=(\Sy_m \times \Sy_k) \cap G$, with $n=m+k$ and  $m>k$ (intransitive case);
\item $G_{v}=(\Sy_m \wr \Sy_k) \cap G$,  with $n=mk$, and $m>1 $ and $k>1$ (imprimitive case);
\item $G_{v}=\AGL_k(p) \cap G$,  with $n=p^k$  (affine case);
\item $G_{v}= (S^k{:}(\Out(S)\times \Sy_k)) \cap G$, with $S$  a nonabelian simple group, $k \geq 2$ and $n=|S|^{k-1}$ (diagonal case);
\item $G_{v}=(\Sy_m \times \Sy_k) \cap G$,  with $n=m^k$, $m \geq 5$ and $k \geq 2$ (wreath case);
\item $S \leq G_{v} \leq \Aut(S)$,  with  $S \neq \A_n$ a nonabelian simple group, and $G_{v}$ acting primitively on $\{ 1,2,\dots,n\}$ (almost simple case).
\end{enumerate}

\begin{lemma}[{\cite[Lemmas~3.2 and 3.3]{PWY2020}}]
The intransitive case and the affine case are impossible.
\end{lemma}

\begin{lemma}
Suppose  $V(\Ga)\leq 30758154560$.
Then the imprimitive case is impossible.
\end{lemma}
\begin{proof}
Suppose that  $G_{v}$ satisfies the imprimitive case.
% that is, $G_{v}=(\Sy_m \wr \Sy_k) \cap G$, with $n=mk$ and $m,k>1$.
Then $|V(\Ga)|=|\Sy_n|/|\Sy_m \wr \Sy_k|$, where $n=mk$ with $m>1$ and $k>1$.
For integers $m>1$ and $k>1$, set
\[
f(m,k)=\frac{(mk)!}{(m!)^kk!}.
\]
Notice that $f(m+1,k)>f(m,k)$ and $f(m,k+1)>f(m,k)$, because
\begin{align*}
\frac{f(m+1,k)}{f(m,k)}&=\frac{((m+1)k)! (m!)^kk!}{(mk)!((m+1)!)^kk!}=\frac{\prod_{i=1}^{k}(mk+i)}{(m+1)^k}>1, \text{ and }\\
\frac{f(m,k+1)}{f(m,k)}&=\frac{(m(k+1))!(m!)^kk! }{(mk)!(m!)^{k+1}(k+1)!}=\frac{\prod_{i=1}^{m}(mk+i)}{m!(k+1)}>1.
\end{align*}
By direct computation,
\begin{align*}
&f(20,2),f(10,3),f(6,4),f(5,5),f(4,6),f(3,8),f(2,12)>30758154560,\\
&f(19,2),f(9,3),f(5,4),f(4,5),f(3,7),f(2,11)<30758154560.
\end{align*}
Therefore, by the assumption $|V(\Ga)|\leq 30758154560$, one of the following holds:
\begin{itemize}
\item $k=2$ and $m\in\{3,4,\dots,19\}$ (noticing that $mk=n\geq 5$);
\item $k=3$ and $m\in \{2,3,\dots,9\}$;
\item $k=4$ and $m\in \{2,3,4,5\}$;
\item $k=5$ and $m\in \{2, 3, 4\}$;
\item $k\in \{6,7\}$ and $m \in \{2,3\}$;
\item $k\in \{8,9,10,11\}$ and $m=2$.
\end{itemize}
For the above candidates of $(k,m)$, when $k=2$ and $m\leq 10$, or $k=3$ and $m\leq 5$, or $k\geq4$, computation in {\sc Magma}~\cite{Magma} shows that there exists no factorization $G_v=G_{uv}G_{vw}$ with $G_{uv}$ conjugate to $G_{vw}$ in $G$ and $|G_v|/|G_{uv}|\geq 3$, and hence these candidates are impossible.
However, the authors' computer is not able to carry out the same computation when $k=2$ and $m\in\{11,\dots,19\}$, or $k=3$ and $m\in\{6,\dots,9\}$. So we apply  Lemma~\ref{lm:imprimitiveaction} to handle these candidates.

We view $V(\Ga)$ as the set of partitions of $\{1,2,\dots,n\}$ with $k$ blocks of size $m$.
Let $v=\{V_1,\dots,V_k\}$, $w=\{W_1,\dots,W_k\}$ and $u=\{U_1,\dots,U_k\}$ such that $U_i^g=V_i$ and $V_i^g=W_i$ for every $i\in \{1,2,\dots k\}$.
Let $A=(|V_i \cap W_j|)_{k\times k}$.
% and $B=(|V_i \cap U_j|)_{k\times k}$.
%Then $B=A^\mathsf{T}$.
Since $v\to w$ and $\to$ is antisymmetric, the arc $(v,w)$ cannot be mapped to $(w,v)$ by any element of $G$.
In other words, $v$ and $w$ are not interchanged by any element in $G$.
Therefore, parts (b)--(d) of Lemma~\ref{lm:imprimitiveaction} imply that one of the following occurs:
\begin{enumerate}[\rm (I)]
\item $G=\A_n$ or $\Sy_n$, and there exist no permutation matrices $P$ and $Q$ such that $A^{\mathsf{T}}=P AQ$;
\item $G=\A_n$, and there exist permutation matrices $P$ and $Q$ such that $A^{\mathsf{T}}=P AQ$ but $(\Sy_n)_{vw} \leq \A_n$. In particular, $m\leq k$.
%$|V_i\cap W_j|\leq1$ for all $i,j\in\{1,\dots,k\}$.
\end{enumerate}
Notice that all entries of $A$ are non-negative integers, and the sum of entries in every row and in every column  of $A$ equals to $m$.

Suppose that $k=2$ and $m\in\{11,\dots,19\}$. Then
\[ A=\begin{pmatrix}
a &b \\
b &a
\end{pmatrix} \text{ for some }  a,b  \in\{0,1,\dots,m\} \text{ such that } a+b=m.\]
Clearly, case (I) is not possible as $A^{\mathsf{T}}=A$. However,
%since $|V_1 \cap W_1|+|V_1 \cap W_2|=a+b=m\geq 3$, at least one of $|V_1 \cap W_1|$ or $|V_1 \cap W_2|$ is larger than $1$, and so case~(II) is not possible either.
case~(II) is not possible either as $m>k$.

Now suppose that $k=3$ and $m\in\{6,\dots,9\}$.
%Since $|V_1 \cap W_1|+|V_1 \cap W_2|+|V_1 \cap W_3|=m\geq 6$, at least one of $|V_1 \cap W_1|$, $|V_1 \cap W_2|$ or $|V_1 \cap W_3|$ is larger than $1$,.
%Hence case~(II) is impossible.
Since $m>k$,  case~(II) is not possible.
Searching such matrices $A=(|V_i \cap W_j|)_{3\times 3}$ satisfying~(I), we obtain that one of the following holds:
\begin{itemize}
\item $m=6$, and $A=\begin{pmatrix}
0 &2 &4\\
3 &2 &1\\
3 &2& 1
\end{pmatrix}$ or its transpose;
\item $m=7$, and $A=\begin{pmatrix}
0 &2 &5\\
3 &3 &1\\
4 &2& 1
\end{pmatrix}$ or its transpose;
\item $m=8$, and $A=\begin{pmatrix}
0 &3 &5\\
4 &2 &2\\
4 &3& 1
\end{pmatrix}$ or its transpose.
\end{itemize}
Recall that $G_v=\Sy_m^3 \wr \Sy_3 \cap G$.
Let $N=\Sy_m^3\cap G$ be the base group of $G_{v}$.
From the factorization $G_{v}=G_{uv}G_{vw}$, we see that $G_{v}/N= (G_{uv}N/N)(G_{vw}N/N)$.
Since $\A_m^3 \wr \A_3 \leq G_v$, it follows that $\A_3 \leq G_{v}/N \leq \Sy_3$, and hence $G_{uv}N/N \geq \A_3$ or $G_{vw}N/N  \geq \A_3$.
Recall $A=(|V_i \cap W_j|)_{k\times k}$.
The group  $G_{vw}N/N $ acts on the set of rows of $A$ by permuting rows.
Since $U_i^g=V_i$ and $V_i^g=W_i$ for every $i\in \{1,2,\dots k\}$, it follows that $|V_i \cap U_j|=|W_i^{g^{-1}} \cap V_i^{g^{-1}}|=|W_i \cap V_j|$ and hence $(|V_i \cap U_j|)_{k\times k}=A^{\mathsf{T}}$.
This implies that $G_{vw}N/N $ permutes columns of $A$.
For the above three candidates for $m$, we see that neither $G_{uv}N/N \geq \A_3$ nor $G_{vw}N/N  \geq \A_3$, which is a contradiction.
%We claim that $G_v/N\cong \Sy_k$.
%It is clear when $G=\Sy_n$.
%Suppose for proving the claim that $G=\A_n$.
%Since $m\geq 2$, $\Sy_m^k$ contains odd permutations.
%It follows that  $\Sy_m^kG=\Sy_n$, and so $G_v/N\cong \Sy_m^kG/\Sy_m^k=\Sy_n/\Sy_m^k \cong \Sy_k$, proving the claim.
%Since $G_{v}/N\cong \Sy_k$, at least one of $G_{uv}N/N$ and $G_{vw}N/N$ is transitive on $\{1,2,\dots,k\}$ (see~\cite[1.3,1.4]{WW1980} for a proof).
%In the case (b.1), we have $G_{vw}N/N\cong \Sy_2$ (for the second and the third rows of $A$ can be swapped) and  $G_{vu}N/N = 1 $ (for   three rows  of $(|V_i \cap U_j|)_{k\times k}=A^{\mathsf{T}}$ are pairwise distinct), and hence $G_{v}/N \neq  (G_{uv}N/N)(G_{vw}N/N)$, a contradiction.
%For case (b.2) and (b.3), we have  $G_{uv}N/N=G_{vw}N/N = 1 $, still yielding a contradiction.
\end{proof}

\begin{lemma}
Suppose $V(\Ga)\leq30758154560$.
Then the  diagonal case, the wreath case and the almost simple case are impossible.
\end{lemma}

\begin{proof}
Suppose that $G_v$ satisfies the diagonal case.
%, that is,$G_{v}= (S^k{:}(\Out(S)\times \Sy_k)) \cap G$, with $S$  a nonabelian simple group, $k \geq 2$ and $n=|S|^{k-1}$.
Then $|V(\Ga)|=|\Sy_n|/|S^k{:}(\Out(S)\times\Sy_k)|$.
Since $S$ is a nonabelian simple group, we have $|S|\geq 60$ and $|\Out(S)|<|S|$. Hence
\[
|V(\Ga)|=\frac{|\Sy_n|}{|S^k||\Out(S)||\Sy_k|}>\frac{(|S|^{k-1})!}{|S|^{k+1}k!}\geq\frac{(|S|)!}{2|S|^3}>30758154560,
\]
a contradiction.

Suppose that $G_v$ satisfies the wreath case.
%that is, $G_{v}=(\Sy_m \wr \Sy_k) \cap G$, with $n=m^k$, $m\geq 5$ and $k\geq 2$.
Then $|V(\Ga)|=|\Sy_n|/|\Sy_m\wr\Sy_k|$, where $n=m^k$ with $m \geq 5$ and $k \geq 2$. It follows that
%\[
%|V(\Ga)|=|G|/|G_v|=|\Sy_n|/|\Sy_m \wr \Sy_k|= (m^k)!/(m!)^k k!.
%\]
\[
|V(\Ga)|=\frac{(m^k)!}{(m!)^kk!}\geq\frac{(m^2)!}{2(m!)^2}\geq\frac{(5^2)!}{2(5!)^2}>30758154560,
\]
still a contradiction.

Finally, suppose that $G_v$ satisfies the almost simple case.
% that is, $S\leq G_{v} \leq \Aut(S)$, with $S $ a non-abelian simple group, $S \neq \A_n$ and $G_{v}$ acting primitively on $\{ 1,2,\dots,n\}$.
As Lemma~\ref{lm:Hsiquasi}(a) asserts, the triple
% \[
% S\in \mathcal{T}_1=\{ \A_6,\M_{12},\PSp_4(2^f),\POm_8^{+}(q)\},
% \]
$(G_{v}, G_{uv}, G_{vw})$ satisfies Lemma~\ref{lm:ASfac}(a)--(d).
Then we conclude from~\cite[Theorem 1.1]{M2002} that either $(G_{v},n)=(\M_{12},12)$, or $|G_{v}|<n^{1+[\log_2(n)]}$.
For the former, $G_{uv}$ and $G_{vw}$ are subgroups of $G_v=\M_{12}$ isomorphic to $\M_{11}$ that are not conjugate in $G_v$.
However, this implies that one of $G_{uv}$ or $G_{vw}$ is a transitive on $\{1,2,\dots,12\}$ but the other stabilizes one point (see~\cite[Table 1]{G2006}), and so $G_{uv}$ and $G_{vw}$ are not conjugate in $G$, a contradiction.
Therefore, $|G_{v}|<n^{1+[\log_2(n)]}$, which implies that
\[
\frac{n!}{2n^{1+[\log_2(n)]}}=\frac{|\A_n|}{n^{1+[\log_2(n)]}}<\frac{|G|}{|G_v|}=|V(\Ga)|\leq30758154560.
\]
This yields $n<20$, and so $P(S)\leq P(G_v)<20$.
Then we derive from~\cite[Table 4]{GMPS2015} that $(G_{v}, G_{uv}, G_{vw})$ does not satisfy Lemma~\ref{lm:ASfac}(c)--(d), that is, $(G_{v}, G_{uv}, G_{vw})$ is one of
\[
(\A_6,\A_5,\A_5),\ \ (\Sy_6,\Sy_5,\Sy_5),\ \ (\M_{12},\M_{11},\M_{11}),
\]
as in Lemma~\ref{lm:ASfac}(a)--(b).
Since $(G_{v},n)\neq(\M_{12},12)$, checking maximal subgroups of $\A_6$ and $\M_{12}$ in Atlas~\cite{Atlas}, we obtain that $ \A_6 \leq G_{v} \leq  \Sy_6$ with $n \in\{6,15\}$.
% From Lemma \ref{lm:ASfac}(a) we see that $(G_v,G_{uv},G_{vw})$ is either $(\A_6,\A_5,\A_5)$ or $(\Sy_6,\Sy_5,\Sy_5)$.
However, computation in {\sc Magma}~\cite{Magma} shows that $G_{uv}$ and $G_{vw}$ are not conjugate in $\Sy_n$, again a contradiction.
\end{proof}

\subsection{Sporadic simple groups}
\ \vspace{1mm}

In this subsection, we suppose that $T$ is one of the $26$  sporadic simple  groups.
The readers may see~\cite{W2017} for the list of maximal subgroups of $G$.
For most $G$, the maximal subgroups can be constructed in {\sc Magma}~\cite{Magma} using the command \textsf{MaximalSubgroups}.
For some other $G$, generators of the maximal subgroups of $G$ are given in Web Atlas~\cite{AtlasOnline}, so we can still construct them in {\sc Magma}~\cite{Magma}.

For the sporadic group $\HN$, generators for its maximal subgroups are given in~\cite{AtlasOnline} except $2^{3+2+6}.(3 \times  \PSL_3(2))$, $3^4{:}2.(\A_4 \times \A_4).4$, $(\A_6 \times \A_6).\D_8$. Howerver, we can construct these three maximal subgroups in  {\sc Magma}~\cite{Magma} as follows.
The first two can be constructed using the method in~\cite[p.318]{BOW2010}.
For example, to construct $2^{3+2+6}.(3 \times  \PSL_3(2))$, we  first construct a Sylow $2$-subgroup $P$ of $\HN$ (note that $|2^{3+2+6}.(3 \times  \PSL_3(2))|_2=2^{14}=|\HN|_2$), and then compute normalizers in $\HN$ of normal $2$-subgroups with order $2^{11}$ in $P$.
A normalizer with order $|2^{3+2+6}.(3 \times  \PSL_3(2))| $  is the desired maximal group $2^{3+2+6}.(3 \times  \PSL_3(2))$.
The maximal subgroup $(\A_6 \times \A_6).\D_8$ has a normal subgroup $(\A_6 \times \A_6).2^2 $ contained in  $\A_{12}$ while generators of $\A_{12}$ are given in~\cite{AtlasOnline}.
So we  construct the group $(\A_6 \times \A_6)^{.}2^2 $ in $\A_{12}$ first and then compute its normalizer, which is the desired $(\A_6 \times \A_6){.}\D_8$.

\begin{lemma}\label{lem4.4}
The group $T$ is not any of the following:
\[ \begin{split}
&  \M_{11}, \M_{12}, \M_{22}, \M_{23},\M_{24}, \J_1, \HS, \J_2,  \McL,\Suz, \J_3,\Co_3,\Co_2, \He, \Fi_{22}, \Ru,\\
&  \Fi_{23}, \Th,\J_4,\Ly,\HN,\ON.
 \end{split} \]
\end{lemma}

\begin{proof}
For $T$ in the first row, computation in {\sc Magma}~\cite{Magma} shows that there exists no factorization $G_{v}=G_{uv}G_{vw}$ with  $|G_v|/|G_{uv}|\geq 3$.

Suppose that $T=\Fi_{23}$. Then $G=\Fi_{23}$ as $|\Out(T)|=1$. If $G_v$ is a maximal subgroup of $G$ other than $3^{1+8}.2^{1+6}.3^{1+2}.2\Sy_4$, then {\sc Magma}~\cite{Magma} shows that there exists no factorization $G_{v}=G_{uv}G_{vw}$ with  $|G_v|/|G_{uv}|\geq 3$. Now assume that $G_{v}=3^{1+8}.2^{1+6}.3^{1+2}.2\Sy_4$.
In this case, $|G_{v}|=2^{11}\cdot 3^{13}$ and so $|G_{uv}|$ is a multiple of $2^6\cdot 3^7$ by~\eqref{EqnDivisor}.
However, it is difficult  to compute all subgroups with order divisible by $2^6\cdot 3^7$ in $G_{v}$.
We let $N$ be the normal subgroup of $G_v$ isomorphic to $3^{1+8}$ and consider the factorization $ G_{v}/N=(G_{uv}N/N)(G_{vw}N/N)$.
Now $G_v/N\cong 2^{1+6}.3^{1+2}.2\Sy_4 $ and both $G_{uv}N/N$ and $G_{vw}N/N$ have order divisible by $2^6$. Let $X=\bO_3(G_{uv}N/N)$ and $Y=\bO_3(G_{uv}N/N)$. Noticing that $N$ is a normal $3$-subgroup of $G_v$, we obtain
\[
(G_{uv}N/N)/X\cong(G_{uv}/(G_{uv}\cap N))/\bO_3(G_{uv}/(G_{uv}\cap N))\cong G_{uv}/\bO_3(G_{uv}).
\]
Similarly, $(G_{vw}N/N)/Y \cong G_{vw}/\bO_3(G_{vw})$.
Then we derive from $G_{uv}\cong G_{vw}$ that
\[
(G_{uv}N/N)/X\cong G_{uv}/\bO_3(G_{uv})\cong G_{vw}/\bO_3(G_{vw})\cong(G_{vw}N/N)/Y.
\]
By computing  factorizations of $ G_v/N$ in {\sc Magma}~\cite{Magma} with the two factors $G_{uv}N/N$ and $G_{vw}N/N$ satisfying $(G_{uv}N/N)/\bO_3((G_{uv}N/N)\cong(G_{vw}N/N)/\bO_3((G_{vw}N/N)$, we obtain  $G_{uv}N/N=G_{vw}N/N=G_{v}/N$. Hence $2^{11}$ divides $|G_{vw}|$, which in conjunction with~\eqref{EqnDivisor} implies that $|G_{uv}|=|G_{vw}|$ is divisible by $2^{11}\cdot 3^7$.
However, computation in {\sc Magma}~\cite{Magma} shows that there exists no  factorization $G_{v}=G_{uv}G_{vw}$ with $|G_v|/|G_{uv}|\geq 3$.

For $T\in\{\Th,\J_4,\Ly,\HN,\ON\}$, computation in {\sc Magma}~\cite{Magma} shows that there is no  factorization $G_{v}=G_{uv}G_{vw}$ with $|G_v|/|G_{uv}|\geq 3$, except when $G=\HN{:}2$ and $G_v=(\Sy_6 \times \Sy_6){.}2^2$. In this exception,  $G$ and $G_v$ are constructed  from $133 \times 133$ matrices over $\mathbb{F}_5$, and $G_v$ has homogeneous factorizations $G_v=KL$ with  $|G_v:K|\geq 3$.
We rule out these homogeneous factorizations by showing that $K$ and $L$ are not conjugate in $G$.
(The group $G$ is too large for the command {\sf IsConjugate} in {\sc Magma}~\cite{Magma} to work).
% We compute the conjugacy classes of $K$ and $L$ to test whether $K$ and $L$  are conjugate in $G$.
In fact, in each of those homogeneous factorizations $G_v=KL$, we find that $K$ has an element of order $2$ or $3$ that is not similar to any element of the same order in $L$ (verified by command {\sf IsSimilar} in {\sc Magma}~\cite{Magma}).
This implies $K$ and $L$ are not conjugate in $G$.
\end{proof}

% Now we remain four simple  sporadic groups $\mathbb{M},\mathbb{B}, \Fi_{24}', \Co_1$. There are few information of generators of their maximal subgroups in Web Atlas~\cite{AtlasOnline}.
 \begin{lemma}
Suppose $|V(\Ga)|\leq 30758154560$.
Then $T$ is not $\Co_1$, $\Fi_{24}'$, $\mathbb{B}$ or $\mathbb{M}$.
\end{lemma}
\begin{proof}
By~\cite{W2017}, we see that the Monster group $P(\mathbb{M})>30758154560$, hence  $T \neq \mathbb{M}$.

Suppose $T=\mathbb{B}$.
Then $G=T$ as $\Out(T)=1$, and $2.{}^2\E_6(2){:}2 $ is the unique candidate  for $G_{v} $ satisfying  $|G|/|G_{v} | \leq 30758154560  $.
However, for this candidate, $G_v^{(\infty)}$ is quasisimple but ${}^2\E_6(2) \notin \mathcal{T}_1 \cup \mathcal{T}_2$, contradicting  Lemma~\ref{lm:Hsiquasi}(b).

Suppose that $T=\Fi_{24}'$.
Then $G=T$ or $T.2$, and  the candidates  for $T_{v} $ satisfying  $|T|/|T_{v} | \leq 30758154560  $ are $\Fi_{23}$, $2.\Fi_{22}{:}2$ and $(3 \times \POm^{+}_{8}(3){:}3){:}2$.
For these  candidates, $G_v^{(\infty)}$ is quasisimple.
By Lemma~\ref{lm:Hsiquasi}(b), the candidates $\Fi_{23}$ and $2.\Fi_{22}{:}2$ are ruled out immediately, and for $T_v=(3 \times \POm^{+}_{8}(3){:}3){:}2$, the factorization $\overline{G_v}=\overline{G_{uv}}\,\overline{G_{vw}}$ satisfies Lemma~\ref{lm:ASfac}.
However, when  $T_v=(3 \times \POm^{+}_{8}(3){:}3){:}2$,  computation in  {\sc Magma}~\cite{Magma} shows that $ \overline{G_v}=\POm^{+}_{8}(3){:}\Sy_3\nleq \mathrm{P\Gamma O}_8^{+}(3)$, contradicting  Lemma~\ref{lm:ASfac}.

Suppose that $T=\Co_1$. Then $G=T$ as $\Out(T)=1$, and  the candidates for $G_{v}$ satisfying  $|G|/|G_{v} | \leq 30758154560  $ are
   \[
   \begin{split}
   & \Co_2, 3.\Suz.2,  \Co_3,\PSU_6(2){:}\Sy_3, (\A_4 \times  \G_2(4)){:}2, \\
   & 2^{1+8}.\POm_8^{+}(2), 2^{11}{:}\M_{24},
 2^{2+12}{:}(\A_8 \times  \Sy_3), 2^{4+12}.(\Sy_3 \times 3.\Sy_6), 3^2.\PSU_4(3).\D_8, 3^6{:}2.\M_{12}.
 \end{split} \]
Candidates in the first row are ruled out by Lemma \ref{lm:Hsiquasi}.
% The groups $2^{1+8}.\POm_8^{+}(2)$ and $2^{11}{:}\M_{24}$ are constructed by using information of generators  in~\cite{AtlasOnline}.
% The groups $2^{2+12}{:}(\A_8 \times  \Sy_3)$, $2^{4+12}.(\Sy_3 \times 3.\Sy_6)$, $3^2.\PSU_4(3).\D_8$ and $3^6{:}2.\M_{12} $ are  constructed by computing  normalizers of some $2$- or $3$-subgroups.
For candidates in the second row, computation  in {\sc Magma}~\cite{Magma} shows  that there is no  factorization $G_{v}=G_{uv}G_{vw}$ with $|G_v|/|G_{uv}|\geq 3$.
\end{proof}

 \subsection{Exceptional groups of Lie type}
 \ \vspace{1mm}

Throughout this subsection,  let $T$ be an exceptional groups of Lie type and let $q=p^f$, where $p$ is prime.

%By Chen, Guidici and Praeger~\cite{CGP2021+}, the case $T \in\{ {}^2B_2(q),{}^2G_2(q) \} $ is impossible and so we have the next lemma.
\begin{lemma}[{\cite{CGP2021+}}]\label{lm:2B22G2}
The group $T \notin\{ {}^2\B_2(q),{}^2\G_2(q)  \}$.
\end{lemma}

%Now that $T$ cannot be ${}^2\B_2(q)$ or ${}^2\G_2(q)$ under Hypothesis~\ref{hy:1}.
 %In the following three lemmas, we exclude the possibilities for $T$ being the remaining exceptional groups of Lie type under Hypothesis~\ref{hy:1} and the condition $|V(\Ga)|\leq 30758154560$.

 In the following three lemmas, we exclude the possibilities for $T$ being the remaining exceptional groups of Lie type under the condition $|V(\Ga)|\leq 30758154560$.

\begin{lemma}\label{lm:3D4}
Suppose   $|V(\Ga)|\leq 30758154560  $.
Then $T \neq {}^3\D_4(q) $.
\end{lemma}
\begin{proof}  Suppose for a contradiction that $T= {}^3\D_4(q)$.
From the list of maximal subgroups of ${}^3\D_4(q) $ (see~\cite{K1988} or~\cite[Table~8.51]{BHRD2013}),  we see that $G_{v}=T_v.\mathcal{O}$, where $\mathcal{O}=G/T \leq \ZZ_{3f}$, and $T_{v}$ satisfies one of the following:
\begin{enumerate}[\rm (a)]
\item $T_{v}\in \{[q^{9}]{:}(\SL_2(q^3)\circ (q-1)).(2,q-1),  [q^{11}]{:}(\SL_2(q )\circ (q^3-1)).(2,q-1)\}$ is a maximal parabolic subgroup;
\item $T_{v}= \G_2(q)$, or $\PGL_3(q)$ with $q\equiv 1(\mod 3)$, or $\PGU_3(q)$ with $q\equiv 2(\mod 3)$ and $q\neq 2$, or ${}^3\D_4(q_0)$ with $q=q_0^a$ and $a\neq 3$ prime;
\item $T_{v}=((q^2+q+1)\circ \PSL_3(q)).(3,q^2+q+1).2$;
\item $T_{v}=((q^2-q+1)\circ \PSU_3(q)).(3,q^2-q+1).2$;
\item $T_{v}= (\SL_2(q^3 )\circ \SL_2(q )).(q-1,2)$;
\item $T_{v} =(q^4-q^2+1).4$;
\item $T_{v} \in \{  (q^2+q+1)^2.\SL_2(3),\  (q^2-q+1)^2.\SL_2(3) \}$.
\end{enumerate}

Case (a) is impossible by Lemma~\ref{lm:PJmaximal}. Case (b) is ruled out by  Lemma \ref{lm:Hsiquasi}(a).

Suppose that (c) occurs.
By  Lemma \ref{lm:Hsiquasi}(b),  since $G_v^{(\infty)}$ is quasisimple,  the factorization $\overline{G_{v}}=\overline{G_{uv}}\,\overline{G_{vw}}$ satisfies  Lemma~$\ref{lm:ASfac2}$.
From Lemma~$\ref{lm:ASfac2}$ we obtain that $q\in \{2,3,4,8\}$.
Suppose $q=2$. Then $G=T$ and $G_{v}=(\ZZ_7 \times \PSL_3(2)).2$.
Now $\overline{G_{v}}$ is an almost simple group with socle $\PSL_3(2)\cong\PSL_2(7)$ and $|\Rad(G_{v})|_2\leq2$.
Without loss of generality, assume that $|\overline{G_{uv}}|_2 \geq |\overline{G_{vw}}|_2$.
Then we deduce from $|G_{uv}|=|G_{vw}|$ that $|\overline{G_{uv}}|_2/|\overline{G_{vw}}|_2\leq2$.
However, the factorization of $\overline{G_{v}}$ given in Lemma \ref{lm:ASfac2} shows that $|\overline{G_{uv}}|_2 /|\overline{G_{vw}}|_2=2^3$, a contradiction.
The case $q\in \{3,4,8\}$ is ruled out similarly.

Suppose that  (d) happens.
By  Lemma \ref{lm:Hsiquasi}(b),  the factorization $\overline{G_{v}}=\overline{G_{uv}}\,\overline{G_{vw}}$ satisfies  Lemma~$\ref{lm:ASfac2}$. From Lemma~$\ref{lm:ASfac2}$ we see that the only possibility is $q=8$. Considering $|\overline{G_{vw}}|_2$ and $|\overline{G_{uv}}|_2$, we obtain a contradiction along the similar lines as in Case~(c).

Suppose that  (e) appears.
Let $N \unlhd G_v$ such that $ G_{v}/N$ is an almost simple group with socle $\PSL_2(q^3) $. Then $\pi(N) \subseteq \pi(\SL_2(q)) \cup \pi(3f)$.
Since $G_v^{(\infty)}=\SL_2(q^3) \circ \SL_2(q)$, it follows that $G_v$ has a unique normal subgroup $M\cong \SL_2(q^3) $.
Furthermore, $M$ is  the smallest subgroup of $G_v$ that contains a nonsolvable composition factor $\PSL_2(q^3)$.
If both $G_{uv}$ and $G_{vw}$ contain a nonsolvable composition factor $\PSL_2(q^3)$, then it follows from $G_{uv}^g=G_{vw}$ that $M^g=M$,  contradicting Lemma \ref{pro:Gvnormal}.
Therefore, $ G_{uv}N/N$ and $ G_{vw}N/N$ are core-free in $ G_{v}/N$, and so we obtain a factorization
\[
G_{v}/N=(G_{uv}N/N)(G_{vw}N/N)
\]
of the almost simple group $G_v/N$ with core-free factors.
% If $q=4$, then $13\in\pi(G_v)$ and $13 \notin \pi(N)$, implying that both $|G_{uv}N/N|$ and $|G_{vw}N/N|$ are divisible by $13$.
% However, computation in {\sc Magma}~\cite{Magma} shows that no almost simple group with socle $\PSL_2(4^3)$ has a factorization with the orders of both core-free factors divisible by $13$.
% Thus $q\neq 4$.
Take $r\in \ppd(p,6f)\subset\pi(G_v)$. Then since $r>3f$ and $|\SL_2(q)|=q(q^2-1)$, we have $r \notin \pi(N)$.
Thus it follows from $G_v=G_{uv}G_{vw}$ that $|G_{uv}N/N|$ and $|G_{vw}N/N|$ are both divisible by $r$.
However, by~\cite[Theorem~A]{LPS1990}), no almost simple group with socle $\PSL_2(q^3)$ has a factorization with the orders of both core-free factors divisible by $r$, a contradiction.

Suppose  that  (f) occurs.
Note that $q^6+1=(q^2+1)(q^4-q^2+1)$.
Let $r\in \ppd(p,12f)$. Then $r \mid (q^6+1)$  but  $r \nmid (q^4-1)$, and so $r \mid (q^4-q^2+1)$.
Moreover,   $r \nmid (3f)$ as $r>12f$.
Thus $G_{v}$ has a unique characteristic subgroup $M\cong \ZZ_r \leq \ZZ_{q^4-q^2+1}$.
From the homogeneous factorization $G_{v}=G_{uv}G_{vw}$, we conclude that both $G_{uv}$ and $G_{vw}$ contain  $M$.
Since $G_{uv}^g=G_{vw}$, it follows  that $M^g=M$,   contradicting Lemma~\ref{pro:Gvnormal}.

Finally, for Case (g), the candidates for $q$ such that $|T|/|T_{v}|=|V(\Ga)| \leq 30758154560 $ are $2$ and $3$.
For these candidates, computation in {\sc Magma}~\cite{Magma} shows that there is no homogeneous factorization $G_{v}=G_{uv}G_{vw}$, a contradiction.
%Now  $G$ and $G_v$ are constructed in {\sc Magma}~\cite{Magma} with commands {\sf AutomorphismGroupSimpleGroup} and {\sf MaximalSubgroups}, respectively.
%Computation in {\sc Magma}~\cite{Magma} shows that there exists no  factorization $G_{v}=G_{uv}G_{vw}$ with  $|G_v|/|G_{uv}|\geq 3$.
 \end{proof}

 \begin{lemma}\label{lm:G2}
Suppose   $|V(\Ga)|\leq 30758154560  $.
Then $T \neq \G_2(q)$.
\end{lemma}
\begin{proof}
Suppose for a contradiction that $T=\G_2(q)$.
Note  that $\Out(T)\cong \ZZ_f$ if $p\neq 3$, while for $p=3$, the group $T$ admits a graph automorphism of order $2$ and $\Out(T)\cong \ZZ_{2f}$.
The reader may see~\cite{C1981} and~\cite{K1988g2} (or~\cite[Table~8.30,~8.41~and~8.42]{BHRD2013}) for the list of maximal subgroups of $G$.

Suppose that $T_{v}$ is a  parabolic subgroup. By Lemma~\ref{lm:PJmaximal} we see that $T_v$ is not a maximal parabolic subgroup.
Thus $T_{v}$ is a non-maximal parabolic subgroup. This implies that $p=3$, $G$ contains a graph automorphism, and $T_{v}=[q^6]{:}(q-1)^2$ is a Borel subgroup of $T$.
Applying Lemma~\ref{lm:PJsuborbits} and Lemma \ref{lm:PJlength}, computation on the Weyl group of $T$ shows that there are four non-self-paired $T$-suborbits relative to $v$, two of length $q^4$ and two of length $q^2$.
Hence $|T_v|/|T_{vw}|=q^4$ or $q^2$. By~\eqref{EqTvroot} we have $|T_{vw}|_p^2\geq|T_v|_p$. Then since $|T_v|_p=q^6$, we obtain $|T_v|/|T_{vw}|=q^2$.
However, we shall show in the next paragraph that the two $T_v$-orbits of length $q^2$ are fused by $G_v$, contradicting the fact that $\Ga$ is $T$-arc-transitive.
% Applying~\eqref{EqBBB}, computation on the roots of $T$ shows that $T_{uvw}=[q^2]{:}(q-1)^2$.
% Then $|T_v||T_{uvw}|=|T_{uv}||T_{vw}|$ and so  $T_v=T_{uv}T_{vw}$, implying that $\Ga$ is $(T,2)$-arc-transitive.
%  and so $\Ga$ does not arise from any of them.

 Let $\Phi=\Phi^{+} \cup \Phi^{-}$ be a root system of $T$, where $\Phi^+=\{b,a,b+a,b+2a,b+3a,2b+3a \}$ with  fundamental roots  $a$ and $b$.
Label roots as follows: \[ \begin{array}{ccccccccccccc }  \hline
 b&a&b{+}a&b{+}2a&b{+}3a&2b{+}3a&{-}b&{-}a&{-}b{-}a&{-}b{-}2a&{-}b{-}a&{-}2b{-}3a  \\
 1&2&3&4&5&6&7&8&9&10&11&12 \\\hline
 \end{array}
\]
Let $T_v$ be the Borel subgroup corresponding to $\Phi$. Note that the Weyl group $W$ of $T$ may be identified with a permutation group on $\Phi$, and a graph automorphism may be identified with a permutation in $\Nor_{\Sym(\Phi)}(W)$ (see~\cite[Section 12.4]{Carter1972}).
By computation in {\sc Magma}~\cite{Magma}, $W$ is generated by the fundamental reflections
\[
(1, 5)(2, 8)(3, 4)(7, 11)(9, 10)\ \text{ and }\ (1, 7)(2, 3)(5, 6)(8, 9)(11, 12).
\]
In $\Nor_{\Sym(\Phi)}(W)$, we may find a graph automorphism
\[
\gamma=(1, 2)(3, 5)(4, 6)(7, 8)(9, 11)(10, 12).\]
According to the action of graph automorphism on root subgroups (see~\cite[Section 12.4]{Carter1972}, we see that  $\gamma$ normalizes $T_v$ and so $\gamma \in G_v$. By Lemma~\ref{lm:PJsuborbits}, those two $T_v$-orbits of length $q^2$ correspond to two non-involutions $g_1,g_2\in D_{\emptyset,\emptyset}=W$ (as $W_\emptyset=1$), where
\[
 g_1=(1, 11, 12, 7, 5, 6)(2, 4, 3, 8, 10, 9) \text{ and } g_2=g_1^{-1}.
\]
Noticing $g_2=\gamma g_1 \gamma$, we conclude that $v^{g_2}=v^{\gamma g_1 \gamma}=v^{g_1\gamma} \in (v^{g_1})^{G_v}$, that is, these two $T_v$-orbits of length $q^2$ are fused by $G_v$, as claimed.
%Notice that $\Ga$ is $T$-arc-transitive and $G$-arc-transitive. This implies $\Ga(v)^{+}=v^{T_vgT_v}=v^{G_vgG_v}$.
%However, $g_2 \notin T_vg_1T_v$ but $g_2 = \gamma g_1 \gamma \in G_vg_1G_v$, implying $v^{g_2} \notin v^{T_vgT_v}$ but $v^{g_2} \in v^{G_vgG_v}$, a contradiction.
%
%
%Note that
%$\Phi^{+} \cap (\Phi^{+})^{g^{-1}}=\{ 2,4,5,6\}$.
%So
%%$1^{g^{-1}}=6 $, $6^{g^{-1}}=5$, $3^{g^{-1}}=4$ and $4^{g^{-1}}=2$.
%\[ T_{vw} =\langle H,X_r \mid r \in \{ b, b+ a,b+2a,2b+3a\}\rangle.\]
%Recall $\gamma \in G_v$  and $ |G_{vw}:T_{vw}|=|G_v:T_v|$.
%So $G_{vw}=\N_{G_v}(T_{vw})$ and  $T_{vw}$ is normalized by some element $h=\phi\gamma x \in  G_{vw} \setminus T_{vw} $, where $x \in T_v$ and $\phi$ is a field automorphism define as in~\cite[p.200]{Carter1972}.
%Let $M=\langle X_r \mid r \in \{ b, b+ a,b+2a,2b+3a\}\rangle$.
%Then $M$ is the largest normal $p$-subgroup of $T_{vw}$, and hence $M=M^{h}=M^{\phi\gamma x}=M^{\gamma x}$ (as $\phi$ normalizes each root subgroup, see~\cite[p.200]{Carter1972}).
%Now $M^{\gamma}=\langle X_r|r \in \{ a, b+3a,b+2a,2b+3a\}\rangle$ as $\{1,3,4,6\}^\gamma=\{ 2,5,4,6 \}$.
%Since $H$ normalizes each root subgroup, $x$ is taken as an element in $U=\langle X_r|r \in \Phi^{+}\rangle$.
%By the structure of $U$ (see~\cite[Theorem 5.3.3]{Carter1972}), $U$ has a normal subgroup $N:=\langle X_{2b+3a},X_{b+3a}\rangle$.
%Since $ M=(M^{\gamma})^x$, $ M/N=(M^{\gamma}/N)^x$.
%However, $|M/N|=q^3$ while $|M^{\gamma}/N|=q^2$, a contradiction.

Among other candidates for $T_v$, we may apply Lemma \ref{lm:Hsiquasi} to rule out $\SL_3(q).2$, $\SU_3(q).2$, $\G_2(q_0)$, $\PSL_2(13)$, $\J_2$, $\PSL_2(8)$, $\PSU_3(3){:}2$, $\PGL_2(q)$ and ${}^2\G_2(q)$.
For the remaining candidates for $T_v$, by the condition $|T|/|T_v|=|V(\Ga)|\leq 30758154560$, one of the following holds:
\begin{enumerate}[\rm (a)]
\item  $T_{v}=(\SL_2(q) \circ \SL_2(q)).2$ with $q\in \{3,5,7,9,11,13,17,19\}$;
\item   $T_{v}=\SL_2(q) \times \SL_2(q)$ with $q\in \{4,8,16\}$;
\item $T_{v}=2^3.\PSL_3(2) $ with $q\in \{3,5,7\}$;
\item  $T_{v}=(q^2\pm q+1).6 $ with $q=9$;
\item $T_{v}= (q\pm 1)^2.\D_{12}$ with $q=9$.
\end{enumerate}

Suppose that (a) or~(b) occurs. For each admissible value of $q$ in these two cases, we construct $T$ in {\sc Magma}~\cite{Magma} by the command {\sf GroupOfLieType} and construct $T_v$ as $\langle X_r \mid r \in \{ b,2a+b,-b,-2a-b\}\rangle$.
Then computation shows that $T_v$ has no subgroups $K$ and $L$ conjugate in $T$ but not in $T_v$ such that $|T_v|$ is divisible by $|KL|$ and $|\Out(T)|$ is divisible by $|T_v|/|KL|$.
This contradicts Lemma~\ref{lm:T2at?}(b)(c).

Suppose that (c) appears. Here $T_{v}=2^3.\PSL_3(2)$ is a non-split extension of $2^3$ by $\PSL_3(2)$ (see~\cite[Table~8.41~and~8.42]{BHRD2013}).
For such a group $2^3.\PSL_3(2)$, computation in {\sc Magma}~\cite{Magma} shows that it has only one conjugacy class of subgroups with order divisible by $2^3\cdot3\cdot7$ and index at least $3$. This implies that $T_{uv}$ is conjugate to $T_{vw}$ in $T_v$, contradicting Lemma \ref{lm:T2at?}(e).

For Case~(d), we have $T_v= 73{:}6$ or $91{:}6$, which can be ruled out by Lemma \ref{pro:Gvnormal}.

Suppose that (e) happens. Now $T_v=8^2.\D_{12}$ or $10^2.\D_{12}$. We construct $T$ in {\sc Magma}~\cite{Magma} by the command {\sf GroupOfLieType} and construct $T_v$ by computing the normalizer of a Sylow $2$-subgroup or $5$-subgroup of $T$ according to $T_v=8^2.\D_{12}$ or $10^2.\D_{12}$, respectively (notice that $|T_v|_2=|T|_2=2^8$ and $|T_v|_5=|T|_5=5^2$). However, computation shows that $T_v$ has no subgroups $K$ and $L$ conjugate in $T$ but not in $T_v$ such that $|T_v|$ is divisible by $|KL|$ and $|\Out(T)|$ is divisible by $|T_v|/|KL|$, contradicting Lemma~\ref{lm:T2at?}(b)(c).
\end{proof}

Following~\cite{AB2015}, a subgroup $X$ is said to be \emph{large} in group $Y$  if $|X|>|Y|^{1/3}$.
Alavi and  Burness~~\cite{AB2015} classified large maximal subgroups of almost simple groups.
Note that $|V(\Ga)|=|T|/|T_v|$. Thus, if  $|V(\Ga)| \leq 30758154560$ and $|T|^{2/3}>30758154560$, then  $|T_v|>|T|^{1/3}$ and so $T_v$ is large  in $T$.

%  The next lemma is consequence of direct computation.
%
%  \begin{lemma}\label{lm:smallT}
%Let $T$ be a simple  group of Lie type. Then, up to isomorphism, $|T|^{2/3} \leq 30758154560$ if and only if $T$ is one of the following groups:
%\begin{itemize}
%\item $\PSL_n(q)$ or $\PSU_n(q)$ with $(n,q)=(2, q\leq 170003)$, $(3,q\leq 97)$,   $(4,q\leq 11)$, $(5,q\leq 4)$,  $(6, 2)$ or $(7, 2)$;
%
%\item $\PSp_n(q)$ with $(n,q)=(4,q\leq 37)$,  $(6,q\leq 5)$,  or $(8, 2)$
%
%\item $\POm_n^{\epsilon}(q)$  with $(n,q)= (7,  3)$, $(7,  5)$, $(8,  2)$, $(8,  3)$ or  $(10,  2)$;
%\item  $G_2(q)$ with $q\leq 13$,  $ {}^2B_2(q)$ with $q\in\{2^3,2^5,2^7,2^9,2^11\}$, $ {}^2G_2(3^3)$ with $q\in\{3^3,3^5 \}$,   $ {}^3D_4(q)$ with $q\leq 4$, $F_4(2)$,  or  ${}^2F_4(2)'$.
%
%\end{itemize}
%\end{lemma}

 \begin{lemma}
Suppose $|V(\Ga)|\leq 30758154560  $.
Then $T \neq \E_8(q)$, $\E_7(q)$, $\E_6(q)$, $\F_4(q)$, ${}^2\E_6(q)$ or ${}^2\F_4(q)'$.
 \end{lemma}
 \begin{proof}
By~\cite{Vasilev1,Vasilev2,Vasilev3} (also see in~\cite[Table 4]{GMPS2015}), the minimum index of subgroups in  $T$ is given in Table~ \ref{tb:degree}.

 \begin{table}[h]
 \caption{Minimum index of subgroups in $T$}
\label{tb:degree}
\[
\begin{array}{l|l|l}\hline
T&P(T)& q \mbox{ such that } P(T)\leq 30758154560\\
\hline
\E_8(q) &(q^{30}-1)(q^{12}+ 1)(q^{10}+ 1)(q^6 + 1)/(q-1) &\mbox{none} \\
\E_7(q) &(q^{14}-1)(q^{9}+ 1)(q^{5}- 1) /(q-1)&2 \\
\E_6(q) &(q^{9}-1)(q^{8}+q^4+ 1) /(q-1)&2,3,4 \\
{}^2\E_6(q) &(q^{12}-1)(q^{6}-q^3+ 1)(q^4+1) /(q-1) &2,3  \\
\F_4(q) &(q^{12}-1)( q^4+ 1) /(q-1)&2,3,4 \\
{}^2\F_4(q)' &(q^6 + 1)(q^3 + 1)(q + 1)&2,8 \\
\hline\end{array}
\]
\end{table}

Since $P(T)\leq|T|/|T_v|=|V(\Ga)|\leq30758154560$,  Table~\ref{tb:degree} implies $T \neq \E_8(q)$. Suppose that $T$ is one of the other five groups.

\vspace{0.1cm}
\underline{$T=\E_7(q)$}.
In this case, $q=2$ and so $G=T$. By the list of maximal subgroups of $T$ in~\cite{BBR2015}, candidates for $T_{v}$ such that $|T|/|T_v|\leq 30758154560$ are maximal parabolic subgroups of type $\D_6$ and $\E_6$,
%$ 2^{1+32}{:} \POm^{+}_{12}(2) $ and $ 2^{27}{:} \E_{6}(2) $,
which are impossible by Lemma \ref{lm:PJmaximal}.

\vspace{0.1cm}
\underline{$T=\E_6(q)$}.
In this case, $q\in \{2,3,4\}$.
Suppose first that $q=2$. Then $G=\E_6(2)$ or $\E_6(2).2$. According to the list of maximal subgroups of $G$ in~\cite{KW1990}, candidates for $T_{v}$ such that $|T|/|T_v|\leq 30758154560$ are  $\F_4(2)$,  maximal parabolic subgroups of type in $\{\D_5, \A_1\A_4,\A_1\A_2\A_2\}$, and non-maximal parabolic subgroups of type in $\{\A_1\A_1\A_2, \D_4\}$.
By Lemma \ref{lm:Hsiquasi}(a), the candidate $\F_4(2)$ is impossible.
By Lemma \ref{lm:PJmaximal}, $T_v$ is not a maximal parabolic subgroup of type $\D_5$.
Applying Lemmas~\ref{lm:PJsuborbits},~\ref{lm:PJlength} and~\ref{lm:PJroots}, {\sc Magma}~\cite{Magma} computation on the Weyl group and roots of $T$ shows that~\eqref{EqTvroot} holds only when $T_v$ is a non-maximal parabolic subgroup of type $\A_1\A_1\A_2$ with $|T_v|/|T_{vw}|=q^7(q+1)(q^2+q+1)$.
Now $T_{v}=[2^{31}]{:}(\PSL_3(2) \times \PSL_2(2)^2)$ and $|T_v|/|T_{vw}|=2^7\cdot3\cdot7$.
It follows that  $|G_v|_7=7$ while $|G_{vw}|_7=1$, contradicting the requirement $|G_{vw}|_7^2\geq|G_v|_7$.

Thus we have $q \in \{3,4\}$. Now $|T|^{2/3}> 30758154560$, which implies that $T_{v}$ is large in $T$.
Applying ~\cite[Theorem 7]{AB2015}, computation in {\sc Magma}~\cite{Magma} shows that the only candidate for $T_v$ with $|T|/|T_v|\leq 30758154560$ is the maximal parabolic subgroup of type $\D_5$, which contradicts Lemma~\ref{lm:PJmaximal}.
%By~\cite[Theorem 7]{AB2015}, $T_{v}$ is either a parabolic subgroup, or in the following:
%\[ \begin{split}
%& (q-1)\D_5(q),\A_1(q)\A_5(q),\F_4(q), \C_4(q) (p\neq 2),\E_6(q^{1/2}), \E_6(q^{1/3}) \mbox{ with } q^{1/3}\equiv 1(\mbox{mod } 3),\\
%& (q-1)^2.D_4(q)  \mbox{ with } q\neq 2, (q^2+q+1).{}^3D_4(q).
%\end{split}
%\]
% By computation, the candidates for $T_v$ such that $|T:T_v|\leq 30758154560$ are only the maximal parabilic subgroup of type $D_5$.
%Computation on the Weyl group of $T$ shows that all $T$-suborbits are self-paired, a contradiction.

\vspace{0.1cm}
\underline{$T={}^2\E_6(q)$}. Now  $q\in \{2,3 \}$.
First suppose $q=2$.
The list of  maximal subgroups of almost simple groups with socle ${}^2\E_6(2)$ is in~\cite{Atlas} (according to~\cite[p.304]{BAtlas}, the list is complete). By~\cite{Atlas}, the candidates for $T_v$ such that $|T|/|T_v|\leq 30758154560$ are  $\F_4(2)$, $\Fi_{22}$, $\POm^{-}_{10}(2)$, and four maximal parabolic subgroups of distinct types. By Lemma~\ref{lm:Hsiquasi}(a) and Lemma~\ref{lm:PJmaximal}, it remains to consider maximal parabolic subgroups of the form $[2^{31}]{:}(\PSL_3(2) \times \A_5)$ or $[2^{29}]{:}(\PSL_3(4) \times \Sy_3)$. Applying Lemmas~\ref{lm:PJsuborbits} and \ref{lm:PJlength}, {\sc Magma}~\cite{Magma} computation on the Weyl group  shows that:
\begin{itemize}
\item if $T_v=[2^{31}]{:}(\PSL_3(2) \times \A_5)$ then there are four  non-self-paired $T$-suborbits,  two of  length $q^{11}(q^2+1)(q^2+q+1)$ and two of length $q^{16}(q^2+1)(q^2+q+1)$;
\item if $T_v=[2^{29}]{:}(\PSL_3(4) \times \Sy_3)$ then there are four  non-self-paired $T$-suborbits,  two of  length $q^{10}(q^6-1)/(q-1)$ and two of  length $q^{14}(q^6-1)/(q-1)$.
\end{itemize}
If $T_v=[2^{31}]{:}(\PSL_3(2) \times \A_5)$, then $|T_v|/|T_{vw}|= 2^{11}\cdot 5\cdot7$ or $2^{16}\cdot 5\cdot7 $, which leads to $|G_v|_7=7$ and $|G_{vw}|_7=1$, contradicting $|G_{vw}|_7^2\geq|G_v|_7$. Similarly, if $T_v=[2^{29}]{:}(\PSL_3(4) \times \Sy_3)$, then $|G_v|_7=7$ and $|G_{vw}|_7=1$, a still a contradiction.

Now we have $q=3$. Then $|T|^{2/3}> 30758154560$, and hence $T_{v}$ is large in $T$.
Applying~\cite[Theorem 7]{AB2015}, we conclude that the only candidate for $T_v$ satisfying $|T|/|T_v|\leq 30758154560$ is the maximal parabolic subgroup of type ${}^2\A_5$, which is impossible by Lemma~\ref{lm:PJmaximal}.

\vspace{0.1cm}
\underline{$T=\F_4(q)$}.
Suppose that $T_v$ is a parabolic subgroup of $T$. Then by Lemma~\ref{lm:PJmaximal}, either $T_v$ is a maximal parabolic subgroup of type $\A_1\A_2$, or $T_v$ is a non-maximal parabolic subgroup of type $\A_1\A_1$ or $\C_2$ and $G_v$ contains a graph automorphism of order $2$.
Applying Lemmas~\ref{lm:PJsuborbits},~\ref{lm:PJlength} and~\ref{lm:PJroots}, {\sc Magma}~\cite{Magma} computation on the Weyl group and roots of $T$ shows that~\eqref{EqTvroot} holds only if
\begin{itemize}
\item $T_v$ is a non-maximal parabolic subgroup of type $\C_2$, and $w$ is in one of two $T_v$-orbits of length $q^4(1+2q^2+2q^3+q^4)$, or in one of two $T_v$-orbits of length $q^{10}$; or
\item $T_v$ is a non-maximal parabolic subgroup  of type $\A_1\A_1$, and $w$ is in one of two $T_v$-orbits of length $q^2(1+2q+q^2)$, or in one of two $T_v$-orbits of length $q^6(1+2q+q^2)$.
\end{itemize}
However, similarly as for $\G_2(q)$ in Lemma~\ref{lm:G2}, computation in {\sc Magma}~\cite{Magma} on the Weyl group  of $T$ shows that the two $T_v$-orbits of the same length are fused by $G_v$, a contradiction. Therefore, $T_v$ is not a parabolic subgroup.

From Table~\ref{tb:degree} we see that $q\in \{2,3,4\}$. If $q=2$, then  computation  in {\sc Magma}~\cite{Magma} shows that there exists no factorization $G_{v}=G_{uv}G_{vw}$ with $|G_v|/|G_{uv}|\geq 3$ such that $G_{uv}$ is conjugate to $G_{vw}$ in $G$.
Hence $q\in \{3,4\}$. It follows that $|T|^{2/3}>30758154560 $, and so $T_v$ is a large subgroup of $T$.
However, appealing to~\cite[Theorem 7]{AB2015}, we find that there is no such $T_v$ satisfying  $|T|/|T_v|\leq 30758154560$.

\vspace{0.1cm}
\underline{$T={}^2\F_4(q)'$}.
Now $q=2$ or $8$.
If $q=2$, then  computation with {\sc Magma}~\cite{Magma} shows that there exists no  such  factorization $G_{v}=G_{uv}G_{vw}$ with  $|G_v|/|G_{uv}|\geq 3$. Therefore $q=8$.
By the list of maximal subgroups of $G$ in~\cite{M1991}, the candidates for $T_{v}$ satisfying $|T|/|T_{v}|\leq 30758154560$ are two maximal parabolic subgroups of distinct types, which are impossible by Lemma~\ref{lm:PJmaximal}.
 \end{proof}

\section{Classical groups}\label{Sec2}

In this section, suppose  Hypothesis~$\ref{hy:1}$, and let $T$ be a classical simple group of Lie type and let $q=p^f$, where $p$ is prime.

 %  The case $T\cong \PSL_2(q)$ is impossible by~\cite[Lemma 4.2,4.3 \& Theorem 5.6]{GLX2019}.
\begin{lemma}[{\cite[Lemma~4.2,~Lemma~4.3~and~Theorem~5.6]{GLX2019}}]\label{lm:Tnotpsl2}
The group $T \neq \PSL_2(q)$.
\end{lemma}

Recall that if $|T|^{2/3} > 30758154560$ and $|V(\Ga)|\leq 30758154560$, then   $T_v$ is a large subgroup of $T$, whence~\cite{AB2015} can be applied to obtain candidates for $T_v$.
The next lemma is obtained by direct computation.
%(see~\cite[Table 5.1.A]{K-Lie} for example)
\begin{lemma}\label{lm:T2/3}
Suppose that $T \neq \PSL_2(q)$ is a classical simple group with $|T|^{2/3} \leq 30758154560$.
Then $T$ is isomorphic to one of the following groups:
\begin{enumerate}[\rm (a)]
\item $\PSL_n( q)$ or $\PSU_n(q)$, where $n=3$ and $q\leq 97$, or $n=4$ and $q\leq 11$, or  $n=5$ and $q\leq 4$, or  $(n,q)=(6, 2)$ or $(7, 2)$;
\item $\PSp_n(q)$, where $n=4$ and $q\leq 37$, or $n=6$ and $q\leq 5$,  or $(n,q)=(8, 2)$;
\item $\POm_n^{\epsilon}(q)$, where $(n,q)=(7,  3)$, $(7,  5)$, $(8,  2)$, $(8,  3)$  or $(10,  2)$.
\end{enumerate}
\end{lemma}

\begin{lemma}\label{lm:notparabolic}
Suppose $|V(\Ga)|\leq 30758154560$. Then $T_v$ is not a parabolic subgroup of $T$.
\end{lemma}

\begin{proof}
Suppose for a contradiction that $T_v$ is a parabolic subgroup of $T$.
First suppose that $T_v$ is a maximal parabolic subgroup.
By Lemma~\ref{lm:PJmaximal} we derive that $T=\POm^{+}_{2m}(q)$ with $m\geq 5$ and $T_v$ is of type $\A_{i}\D_{m-1-i} $ with  $m/2<i<m-1$.
The pairs  $(m,i)$ such that $|T|/|T_v|\leq 30758154560$ for some $q$ are $(5,3)$, $(6,4)$, $(7,4)$ and $(7,5)$. Applying Lemmas~\ref{lm:PJsuborbits} and~\ref{lm:PJroots}, computation on the Weyl group and roots of $T$ in {\sc Magma}~\cite{Magma} shows that~\eqref{EqTvroot} does not hold, a contradiction.  Thus $T_v$ is a non-maximal parabolic subgroup, and so one of the following holds:
\begin{enumerate}[\rm (a)]
\item $T=\PSL_{n}(q)$, the group $T_{v}$ is of type $P_{k,n-k}$ with $k<n/2$ (see \cite[Table 2.2]{BHRD2013}), and $G_v$ contains a graph automorphism of order $2$;
\item $T=\PSp_4(2^f)$ with $f\geq 2$, the group $T_v$ is a Borel subgroup of $T$, and $G_v$ contains a graph automorphism of order $2$;
\item  $T=\POm_8^{+}(q)$, the group $T_v$ is a non-maximal parabolic subgroup of type $\A_1$, and $G_v$ contains a graph automorphism of order $2$ or $3$;
\item $T=\POm_{2m}^{+}(q)$ with $m\geq 4$, the group $T_v$ is a non-parabolic subgroup of type $\A_{m-2}$, and $G_v$ contains a graph automorphism of order $2$.
\end{enumerate}
We deal with these cases one by one in the following, where Case~(b) and Case~(c) are treated  similarly  as for $\G_2(q)$ in Lemma~\ref{lm:G2}.

For Case~(a), the  pairs $(n,k)$ such that  $|T|/|T_{v} |\leq 30758154560$ for some $q$ are as follows:
\begin{itemize}
\item  $k=1$ and $ 3\leq n\leq 17$;
\item  $k=2$ and $ 5\leq n\leq 10$;
\item $k=3$ and $ 7\leq n\leq 9$.
\end{itemize}
By Lemma~\ref{lm:PJsuborbits} and Lemma~\ref{lm:PJroots}, computation on the Weyl group and roots of $T$ in {\sc Magma}~\cite{Magma} shows that~\eqref{EqTvroot} does not hold, a contradiction.

For Case~(b),  computation  in {\sc Magma}~\cite{Magma} shows that there are only two non-self paired $T $-suborbits, while they are fused  by $G_v$,  a contradiction.

For Case~(c), {\sc Magma}~\cite{Magma} computation shows that~\eqref{EqTvroot} holds only when $w$ is in one of six non-self paired  $T$-suborbits  of length $q^4(q+1)$.
However, these six $T_v$-orbits are fused by $T_v.\Sy_3$, where the group $\Sy_3$ is generated by a graph automorphism of order $3$ and a graph automorphism of order $2$. This implies $|w^{G_v}|\neq |w^{T_v}|$,  a contradiction.

For Case (d), the integers $m$ such that  $|T|/|T_{v} |\leq 30758154560$ for some $q$ are $4$, $5$, $6$ and $7$. However, computation  in {\sc Magma}~\cite{Magma} shows that~\eqref{EqTvroot} does not hold, a contradiction.
\end{proof}

\begin{lemma}\label{lm:notsmallT}
The group $T$ is not any of the following groups:
\begin{enumerate}[\rm (a)]
\item $\PSL_n(q)$ or $\PSU_n(q)$, where $n=3$ and $q\leq 47$,  or $n=4$ and $q\leq 11$, or  $n=5$ and $q\leq 4$, or  $(n,q)=(6, 2)$ or $(7, 2)$;
\item $\PSp_n(q)$, where $n=4$ and $q\leq 37$, or $n=6$ and $q\leq 5$,  or $(n,q)=(8, 2)$;
\item $\POm_n^{\epsilon}(q)$, where $(n,q)=(7,  3)$, $(7,  5)$, $(8,  2)$, $(8,  3)$  or $(10,  2)$.
\end{enumerate}
\end{lemma}

\begin{proof}
Suppose that $T$ is one of the groups in (a)--(c). For non-parabolic subgroups $T_v$, computation in {\sc Magma}~\cite{Magma} shows that there exists no  factorization $G_v=G_{uv}G_{vw}$ with $|G_{v}|/|G_{uv}|\geq 3$ such that $G_{uv}$ is conjugate to $G_{vw}$ in $G$. Thus $T_v$ is a parabolic subgroup of $T$. It follows that $|V(\Ga)|=|T|/|T_v|\leq 30758154560$, contradicting Lemma~\ref{lm:notparabolic}
\end{proof}

%For $T=\POm^{\pm}_{10}(2)$, applying Lemma~\ref{lm:Hsiquasi} to deal with the case that $G_v^{(\infty)}$ is quasisimple,  it remains  the following two cases:
%\begin{itemize}
%\item $T=\POm^{+}_{10}(2)$, and $T_v=(\A_5\times\PSU_4(2)){:}2 \in \mathcal{C}_1$;
%\item $T=\POm^{-}_{10}(2)$, and $T_v=(\A_5\times\A_8){:}2 \in \mathcal{C}_1$ or $ (3 \times \Omega^{+}_8(2)){:}2 \in \mathcal{C}_1$.
%\end{itemize}
%We construct $\lefthat T_v $ in {\sc Magma}~\cite{Magma}  with the command  {\sf ClassicalMaximals}. Computation in {\sc Magma}~\cite{Magma} shows that   $ \lefthat T_v $ has no subgroups $K$ and $L$ conjugate in $\lefthat T$ but not in $\lefthat T_v$  such that   $|\lefthat T_v|$ is divisible by $|  K  L|$ and $|\Out(T)|$ is divisible by $|\lefthat T_v|/|\lefthat K \lefthat L|$.
%This contradicts Lemma~\ref{lm:T2at?}(b)(c).
%

The maximal subgroups of classical almost simple groups are divided into nine classes $\mathcal{C}_1$, $\mathcal{C}_2$,
$\dots$, $\mathcal{C}_9$ by Aschbacher's theorem~\cite{As1984}.
The maximal subgroups in classes $\mathcal{C}_1$--$\mathcal{C}_8$ are called  geometric subgroups, which are described in
\cite[Chapter 4]{K-Lie} and summarized in~\cite[Section 2.2]{BHRD2013}.
The maximal subgroups in class $\mathcal{C}_9$ arise from irreducible representations of quasisimple groups and are almost simple.

\begin{lemma}\label{lm:C9}
Suppose $|V(\Ga)|\leq 30758154560$ and $T_{v} \in \mathcal{C}_9$.
Then $T=\PSL_3(49) $ with $T_{v}=\A_6$ and $\Ga \cong \Ga_7$.
\end{lemma}

\begin{proof}
%By Lemma \ref{lm:Tnotpsl2},  $T\neq \PSL_2(q)$ and so $T$ satisfies Lemma \ref{lm:T2/3}.
By Lemma~\ref{lm:Hsiquasi}(a), $\Ga$ is $(T,2)$-arc-transitive, and the factorization $G_v=G_{uv}G_{vw}$ and $T_v=T_{uv}T_{vw}$ satisfy Lemma~\ref{lm:ASfac}. In particular, $\Soc(T_{v}) \in \{\A_6,\M_{12}, \PSp_4(2^f), \POm_8^{+}(q)\}$ and $(T_{v}, T_{uv}, T_{vw})$ satisfies (a)--(d) of Lemma~\ref{lm:ASfac}.

First assume that $|T|^{2/3}> 30758154560 $.
Then $ T_{v}  $ is large in $T$.
According to~\cite[Theorem~7]{AB2015}, the pair $(T,\Soc(T_{v}))$ lies in
\cite[Table~3~and~Table~7]{AB2015}. However, there exists no such pair $(T,\Soc(T_{v}))$ in
\cite[Table~3~and~Table~7]{AB2015}  satisfying the conditions $|T|^{2/3}> 30758154560 $, $\Soc(T_{v}) \in \{\A_6,\M_{12}, \PSp_4(2^f), \POm_8^{+}(q)\}$ and $|T|/|T_{v}| \leq 30758154560$.

Next assume that $|T|^{2/3}\leq 30758154560 $.
Since Lemma \ref{lm:Tnotpsl2} implies that $T$ is a group in Lemma~\ref{lm:T2/3}, we conclude by~\cite{BHRD2013} that the candidates of $(T,T_v)$  with $|T|/|T_{v}|\leq 30758154560$ are:
\begin{itemize}
\item $T=\PSL_3(q)$ with $q\in\{4,19,31,49\}$, and $T_{v} =\A_6$;
\item $T=\PSU_3(q)$ with $q\in\{11,29,41\}$, and $T_{v} =\A_6$;
\item $T=\PSp_4(q)$ with $q\in\{5,17,19,29,31\}$, and $T_{v}=\A_6$;
\item $T=\PSp_4(q)$ with $q\in\{13,37\}$, and $T_{v} =\Sy_6$;
\item $T=\POm_8^{+}(3)$, and $T_{v}=\POm_8^{+}(2)$.
\end{itemize}
Computation in {\sc Magma}~\cite{Magma} shows that $T_{uv}$ and $T_{vw}$ are conjugate in $T$ only when $T=\PSL_3(49) $ with $T_{v}=\A_6$.
Now it remains to show $\Ga \cong \Ga_7$.
By \cite[Table~8.3]{BHRD2013} we see that $T$ has three non-conjugate maximal subgroups $\A_6 \in \mathcal{C}_9$, which are fused by a diagonal automorphism of $T$.
This implies that $\Ga_7$ is isomorphic to some orbital digraph corresponding to a non-self-paired $T$-suborbit $\Delta$ relative to $v$.
Notice that $T_v=\A_6$ has exactly two non-conjugate subgroups $\A_5$, and that $T_{uv}$ is  not conjugate to $T_{vw}$ in $T_v$.
We may take $h\in T$ such that $v^h\in \Delta$ and $T_{vv^h}=T_{vw}$ or $T_{uv}$.

Suppose that $T_{vv^h}=T_{vw}$. Then since $T_{v^{h^{-1}}v}$ is not conjugate to $T_{vv^h}$ in $T_v$, it follows that $T_{v^{h^{-1}}v}$ and $T_{uv} $ are in the same conjugacy class in $T_v$, that is, there exists some $y\in T_v$ with $T_{v^{h^{-1}}v}=T_{uv}^y$. Hence
\[
T_{uv}^{yh}=T_{v^{h^{-1}}v}^{h}=T_{vv^h}=T_{vw}.
\]
By the $T$-arc-transitivity of $\Ga$, there exists some $x \in T$ such that $(u,v)^x=(v,w)$ and so $T_{uv}^{x}=T_{vw}$. This together with $T_{uv}^{yh}=T_{vw}$ implies $yhx^{-1} \in \Nor_T(T_{uv})$.
Computation in {\sc Magma}~\cite{Magma} shows $\Nor_T(T_{uv})=T_{uv}$, which implies that $yhx^{-1}\in T_{uv}\leq T_v$.
As a consequence, $h\in y^{-1}T_v x=T_vx$ and hence $v^h\in v^{T_vx}=\{v^x\}=\{w\}$, that is, $v^h=w$.
Therefore,  the orbital digraph corresponding to $\Delta$ is exactly $\Ga$, and so $\Ga\cong \Ga_7$. The case $T_{vv^h}=T_{uv}$ is treated similarly.
\end{proof}

In the following four subsections, we deal with the four families of classical groups respectively.
Denote by $\lefthat T$ the qausisimple group  $\SL_n(q)$, $\SU_n(q)$, $\Sp_n(q)$ or $\Omega^{\epsilon}_n(q)$, corresponding to $T=\PSL_n(q)$, $\PSU_n(q)$, $\PSp_n(q)$ or $\POm^{\epsilon}_n(q)$, respectively.
For a subgroup $X$ of $T$, denote by $\lefthat X$ the (full) preimage of $X$ in $\lefthat T$.
Recall the factorization $\overline{G_v}=\overline{G_{uv}}\,\overline{G_{vw}}$ in~\eqref{EqnFac}.

\subsection{Linear groups}
\ \vspace{1mm}

%Throughout this subsection, suppose Hypothesis~\ref{hy:1} and that $T=\PSL_n(q)$ with $n\geq 3$.

\begin{lemma} \label{lm:pslC1GLmn-m}
Suppose $T=\PSL_n(q)$ with $n\geq 3$. Then  $T_{v} $ is not a $\mathcal{C}_1$-subgroup  of type $\GL_m(q)\oplus \GL_{n-m}(q) $.
\end{lemma}

\begin{proof}
Suppose for a contradiction that $T_{v} $  is such a group.
Let $k=n-m$ and assume without loss of generality that $m<k$.
Then $\lefthat T_v=(\SL_{m}(q) \times \SL_{k}(q)){:}(q-1)$ (see~\cite[Table 2.3]{BHRD2013}).
Let $N \unlhd G_{v}$ such that $G_{v}/N$ is almost simple with socle $\PSU_k(q)$.
%Then $\pi(N)\subseteq \pi(\SU_m(q)) \cup \pi(q+1) \cup \pi(f)$.
By similar argument as for Case~(e) of Lemma~\ref{lm:G2}, the factors $G_{uv}N/N$ and $G_{vw}N/N$ of the factorization $G_{v}/N=(G_{uv}N/N)(G_{vw}N/N)$ are core-free.
%Note that $|G_v|/|T_v|= |G|/|T| $ divides $2f(n,q-1)$.

Suppose that $n=3$.
Then $m=1$, $k=2$,  and $\lefthat T_v\cong\GL_{2}(q)$. Furthermore, by Lemma~\ref{lm:notsmallT} we have $q> 47$.
Now Lemma~\ref{lm:Hsiquasi}(b) implies that the factorization $\overline{G_v}=\overline{G_{uv}}\,\overline{G_{vw}}$ satisfies Lemma~\ref{lm:ASfac2}.
As a consequence, interchanging $G_{uv}$ and $G_{vw}$ if necessary, we have $|\overline{G_{uv}}|_p\geq q$ and $|\overline{G_{vw}}|_p\leq(2f)_p$. It follows that
\[
q \leq |G_{uv}|_p=|G_{vw}|_p\leq |\overline{G_{vw}}|_p|\Rad(G_v)|_p\leq(2f)_p(2(q-1)(3,q-1))_p=(4f)_p,
\]
which forces $q=4$ or $16$, a contradiction.

Thus we conclude that $n\geq 4$, and so $k\geq 3$. By Lemma~\ref{lm:notsmallT}, $(n,q)$ is not a pair such that  $n=4$ with $q\leq11$, $n=5$ with $q\leq4$, or $q=2$ with $n\in\{6,7\}$.
Since $n<2k$, we have $(k,q)\neq (3,4)$.
If $(k,q)\neq (6,2)$, then $|N|_r=1$ for $r\in \ppd(p,kf)$, which together with $|G_{uv}|_r^2=|G_{vw}|_r^2\geq|G_v|_r$ implies that $r$ divides both $|G_{uv}N/N|$ and $|G_{vw}N/N|$, contradicting~\cite[Theorem A]{LPS1990}.  Thus $(k,q)=(6,2)$, and so $8\leq n\leq 11$. Now $G_{v}/N=\PSL_6(2)$ or $\PSL_6(2).2$.
Computation in {\sc Magma}~\cite{Magma} shows that, interchanging $G_{uv}N/N$ and $G_{vw}N/N$ if necessary, one of the following holds:
\begin{enumerate}[\rm (a)]
\item $G_{vw}N/N \in \{ \PSL_5(2),\,\PSL_5(2).2,\,2^5.\PSL_5(2)\}$, the group $G_{uv}N/N$ has a unique nonsolvable composition factor $K$, and $K\in\{\PSL_2(8),\PSL_3(4),\PSU_3(3),\PSp_6(2)\}$;
\item $G_{vw}N/N=(G_{vw}N/N)^{(\infty)}=2^5.\PSL_5(2)$, and $G_{uv}N/N$ is solvable.
\end{enumerate}
Notice $N\in\{\PSL_m(2),\PSL_m(2).2\}$ and that $G_{vw}=(N\cap G_{vw}).(G_{vw}N/N)$ is isomorphic to $G_{uv}=(N\cap G_{uv}).(G_{uv}N/N)$.
If~(a) holds, then $K$ is a composition factor of $N\cap G_{vw}$ and hence a section of $N$, which is not possible as $m\leq k-1=5$.
Thus~(b) occurs. Then since $G_{uv}N/N$ is solvable, we obtain $G_{uv}^{(\infty)}=(N\cap G_{uv})^{(\infty)}\leq N^{(\infty)}=\PSL_m(2)\leq\PSL_5(2)$ as $m\leq5$. However, $(G_{vw}N/N)^{(\infty)}=2^5.\PSL_5(2)$. This contradicts $G_{uv}\cong G_{vw}$.
\end{proof}

\begin{lemma} \label{lm:psl}
Suppose $T=\PSL_n(q)$ with $n\geq 3$ and $|V(\Ga)|\leq 30758154560$. Then $\Ga\cong \Ga_7$.
\end{lemma}

\begin{proof}
By~\cite[Theorem 5.6]{GLX2019}, $T_{v} \notin \mathcal{C}_i $ for $3\leq i \leq 6$.
If $T_{v} \in \mathcal{C}_9$, then Lemma \ref{lm:C9} asserts that  $\Ga\cong \Ga_7$.
It remains to show that the case  $T_{v} \in \{ \mathcal{C}_1,\mathcal{C}_2, \mathcal{C}_7,\mathcal{C}_8\}$ is impossible.
By Lemma~\ref{lm:notsmallT}, $(n,q)$ is not a pair such that $n=3$ with $q\leq 47$, $n=4$ with $q\leq11$, $n=5$ with $q\leq4$, or $q=2$ with $n\in\{6,7\}$.
%Note that $ |\Out(T)|=2fd$, where $d=(n,q-1)$.

\vspace{0.1cm}
\underline{$T_{v}\in \mathcal{C}_1$}.
Then $T_v$ has type $P_m$, $\GL_m(q)\oplus \GL_{n-m}(q)$ or $P_{m,n-m}$, which are not possible by ~\cite[Lemma 4.2]{GLX2019}, Lemma~\ref{lm:pslC1GLmn-m} and Lemma~\ref{lm:notparabolic}, respectively.

\vspace{0.1cm}
\underline{$T_{v}\in \mathcal{C}_2$}.
Then $\lefthat T_v=\SL_m (q)^t.(q-1)^{t-1}.\Sy_t $ with $mt=n$.
The candidates for the triple $(n,m,q)$ such that $|T|/|T_{v}|\leq 30758154560$ are as follows:
\begin{itemize}
\item $(n,m)=(3,1) $ and $49\leq q\leq 73$;
%\item  $(n,m)=(4,1) $ and $q\leq 9$;
\item  $(n,m)=(4,2) $ and $13\leq q\leq 19$;
%\item $(5,1,2)$, $(5,1,2)$, $(5,1,3)$, $(5,1,4)$, $(6,1,2)$, $(6,2,2)$, $(6,3,2)$, $(6,3,3)$, $(8,4,2)$.
\item $(6,3,3)$, $(8,4,2)$.
\end{itemize}
However, computation in {\sc Magma}~\cite{Magma} for these candidates shows that there exists no factorization $G_{v}=G_{uv}G_{vw}$ with  $|G_v|/|G_{uv}|\geq 3$ such that $G_{uv}$ is conjugate to $G_{vw}$ in $G$.

\vspace{0.1cm}
\underline{$T_{v}\in \mathcal{C}_7$}. Now $n\geq 3^2$ (see~\cite[Table 2.10]{BHRD2013}).
By   Lemma \ref{lm:T2/3}, we see that $|T|^{2/3}> 30758154560$  and so $T_{v}$ is large in $T$.
However, by~\cite[Proposition 4.17]{AB2015}, $T$ has no large subgroup in $\mathcal{C}_7$.

\vspace{0.1cm}
\underline{$T_{v}\in \mathcal{C}_8$}. Then $T_v$ has type $\GU_n(q^{1/2})$, $\Sp_n(q)$ or $\GO^{\epsilon}_n(q)$.
Suppose first that $ T_v$ is of type $\GU_n (q^{1/2})$. Then $\lefthat T_v=\SU_{n}(q^{1/2}).(n,q^{1/2}-1)$ (see~\cite[Table~2.11]{BHRD2013}) and $G_{v}^{(\infty)}= \PSU_{n}(q^{1/2})$.
It follows from Lemma~\ref{lm:Hsiquasi}(b) that the factorization $\overline{G_v}=\overline{G_{uv}}\,\overline{G_{vw}} $ satisfies Lemma~\ref{lm:ASfac2}, and so $(n,q^{1/2})=(3,8)$ or $(4,2)$. From Lemma~\ref{lm:ASfac2} we see that $|\overline{G_{uv}}|_r\neq |\overline{G_{vw}}|_r$ for $r\in\ppd(2,nf/2)$. This together with $|\Rad(G_v)|_r=1$ implies that $|G_{uv}|_r\neq |G_{vw}|_r$, a contradiction.

Suppose next that  $T_v$ is of type $\Sp_n (q)$. Then $n\geq 4$ and $G_{v}^{(\infty)} =\PSp_{n}(q)$.
By Lemma~\ref{lm:Hsiquasi}(b) we derive that  $q \geq 16$ is even and the factorization $\overline{G_v}=\overline{G_{uv}}\,\overline{G_{vw}} $ satisfies Lemma~\ref{lm:ASfac}.
According to Lemmas~\ref{lm:ASfac} and~\ref{lm:T2at?} we see that  $T_{uv}^{(\infty)} \cong T_{vw}^{(\infty)}  \cong \Sp_2(q^2).2$.
Since $|T|/|T_{v}|\leq 30758154560$, the candidates for $q$ are $ 16,32,64$.
For these $q$,  computation in {\sc Magma}~\cite{Magma} shows that there is an element of order $3$ in $T_{uv}^{(\infty)} $ not similar to any element of order $3$ in $T_{vw}^{(\infty)}$.
This implies that $T_{uv}$ and $T_{vw}$ are not conjugate in $T$,  contradicting Lemma~\ref{lm:T2at?}.

Therefore, $T_v$ is of type $\GO^{\epsilon}_n (q)$.
Now  $\lefthat T_v=\SO^{\epsilon}_n(q).(n,q-1)$(see~\cite[Table~2.11]{BHRD2013}) and $q$ is odd.
We distinguish the cases $n=3$, $n=4$ and $n\geq5$, respectively.

First assume that $n=3$. Then  $T_v=\SO_3(q)\cong \PGL_2(q)$.
By Lemma \ref{lm:Hsiquasi}(b), the factorization $\overline{G_v}=\overline{G_{uv}}\,\overline{G_{vw}} $ satisfies Lemma~\ref{lm:ASfac2}.
However, Lemma~\ref{lm:ASfac2} shows $|\overline{G_{uv}}|_p\geq q$ and $|\overline{G_{vw}}|_p\leq f_p$ (interchanging $G_{uv}$ and $G_{vw}$ if necessary), which leads to the contradiction $|G_{uv}|_p\geq q>f_p\geq|G_{vw}|_p$ as $ |\Rad(G_v)|$ is coprime to $p$.

Next assume that $n=4$. If $\epsilon=-$, then $G_v^{(\infty)}\cong \PSL_2(q^2)$, which is not possible by Lemma \ref{lm:Hsiquasi} and Lemma \ref{lm:ASfac2}. Thus $\epsilon=+$. Since $|T|/|T_v|\leq 30758154560$, we have $q=13$. However, computation in {\sc Magma}~\cite{Magma} shows that there exists no factorization $G_v=G_{uv}G_{vw}$ with $|G_{v}|/|G_{uv}|\geq 3$ such that $G_{uv}$ is conjugate to $G_{vw}$ in $G$.

Finally, assume that $n\geq 5$. Note that $(n,q)\neq (5,3)$ or $(6,2)$.
%Note that $\POm_5(q)\cong \PSp_4(q)$, $\POm_6^{-}(q)\cong \PSU_4(q)$ and $\POm_6^{+}(q)\cong \PSL_4(q)$.
By Lemma \ref{lm:Hsiquasi}(b), we conclude that $n=8$ with  $T_v=\PSO^{+}_8(q).2$ and  the factorization $\overline{G_v}=\overline{G_{uv}}\,\overline{G_{vw}} $ satisfies Lemma~\ref{lm:ASfac}. However, there is no such $q$ with $|T|/|T_v|\leq 30758154560$.
 \end{proof}

\subsection{Unitary groups}
\ \vspace{1mm}

%Throughout this subsection, suppose Hypothesis~\ref{hy:1} and that $T=\PSU_n(q)$ with $n\geq 3$.

\begin{lemma} \label{lm:psuC1}
Suppose $T=\PSU_n(q)$ with $n\geq 3$ and $|V(\Ga)|\leq 30758154560$. Then  $T_v$ is not a $\mathcal{C}_1$-subgroup.
\end{lemma}

\begin{proof}
Suppose for a contradiction that  $T_v$ is a $\calC_1$-subgroup.
By Lemma~\ref{lm:PJmaximal}, $T_v$ is not a maximal parabolic subgroup of $T$, and so $T_v$ is of type  $\SU_{m}(q) \perp \SU_{n-m}(q)$.  Let $k=n-m$ and assume without loss of generality that $m<k$. Now $\lefthat T_v=(\SU_{m}(q) \times \SU_{k}(q)).(q+1)$.
Let $N \unlhd G_{v}$ such that $G_{v}/N$ is almost simple with socle $\PSU_k(q)$.
%Then $\pi(N)\subseteq \pi(\SU_m(q)) \cup \pi(q+1) \cup \pi(f)$.
By similar argument as for Case~(e) of Lemma~\ref{lm:G2}, the factors $G_{uv}N/N$ and $G_{vw}N/N$ of the factorization $G_{v}/N=(G_{uv}N/N)(G_{vw}N/N)$ are core-free.
By Lemma~\ref{lm:notsmallT} and the assumption $|V(\Ga)|\leq 30758154560$, one of the following occurs:
\begin{enumerate}[\rm (a)]
\item either $m=1$, or $m=2$ and $q\in \{2,3\}$;
\item $(n,k,q)=(5,3,5)$, $(5,3,7)$, $(6,4,4)$  or $(8,5,2)$.
\end{enumerate}

Suppose that (a) happens. In this case, we conclude that $N=\Rad(G_v)$ and $G_{v}^{(\infty)}=\SU_k(q)$ is quasisimple.
%Note that $T\neq \PSU_3(9)$ by Lemma~\ref{lm:notsmallT}, and hence $(k,q)\neq (2,9)$.
By Lemma \ref{lm:Hsiquasi}(b), the factorization $\overline{G_{v}}=\overline{G_{vw}}\,\overline{G_{vw}} $ satisfies Lemma~\ref{lm:ASfac2}, and so either $(k,q)\in\{(3,8),(4,2)\}$ or $k=2$.  Notice that $n=m+k<2k$.
If $(k,q)\in\{(3,8),(4,2)\}$, then Lemma~\ref{lm:notsmallT} implies $(n,k,q)=(5,3,8)$, but we obtain from Lemma~\ref{lm:ASfac2} that $|G_{uv}|_{17} \neq |G_{vw}|_{17}$, a contradiction. Thus $k=2$, and so $n=3$.
According to Lemma~\ref{lm:ASfac2}, interchanging $G_{uv}$ and $G_{vw}$ if necessary, we have $|\overline{G_{uv}}|_p\geq q$ and $|\overline{G_{vw}}|_p\leq(2f)_p$. It follows that
\[
q \leq |G_{uv}|_p=|G_{vw}|_p\leq |\overline{G_{vw}}|_p|\Rad(G_v)|_p\leq(2f)_p(2(q+1)(3,q+1))_p=(4f)_p,
\]
which forces $q=4$ or $16$, contradicting  Lemma~\ref{lm:notsmallT}.

Suppose that (b) happens.  Let $r\in \ppd(q,2k)$ if $k$ is odd, and let $r\in \ppd(q,k)$ if $k$ is even. Then $|N|_r=1$, which together with $|G_{uv}|_r^2=|G_{vw}|_r^2\geq|G_v|_r$ implies that $r$ divides both $|G_{uv}N/N|$ and $|G_{vw}N/N|$.
However, we see that this is not possible by computing the factorizations of almost simple groups with socle $\PSU_3(5)$, $\PSU_3(7)$, $\PSU_4(4)$ or $\PSU_5(2)$ in {\sc Magma}~\cite{Magma}.
\end{proof}

\begin{lemma} \label{lm:psun=3}
Suppose   $|V(\Ga)|\leq 30758154560$. Then $T\neq \PSU_n(q)$ for $n\geq 3$.
\end{lemma}

\begin{proof}
Suppose for a contradiction  that $T=\PSU_n(q)$ with $n\geq 3$. By Lemmas~\ref{lm:notparabolic},~\ref{lm:notsmallT},~\ref{lm:C9} and~\ref{lm:psuC1}, we see that $T_v\notin\{ \mathcal{C}_1,\mathcal{C}_9\}$ and $(n,q)$ is not a pair such that $n=3$ with $q\leq47$, $n=4$ with $q\leq11$, $n=5$ with $q\leq4$, or $q=2$ with $n\in\{6,7\}$.

First suppose that $n=3$. Then $\lefthat T_v$ lies in~\cite[Table 8.5]{BHRD2013}.
By Lemma~\ref{lm:Hsiquasi}(b) and Lemma~\ref{lm:ASfac2}, we can rule out the $\mathcal{C}_5$-subgroups
\[
(3,q+1) \times \SO_3(q)\ \text{ and }\ \SU_{3}(q^{1/r}).\left(\frac{q+1}{q^{1/r}+1},3\right)\ \text{ for odd prime } r.
\]
If $\lefthat T_v=(q^2-q +1).3 \in \mathcal{C}_3$, then $G_v$ has a normal subgroup $M\cong \ZZ_r$ with $r\in \ppd(p,6f)$ such that $M$ is in both $G_{uv}$ and $G_{vw}$, which implies $M^g=M$, contradicting Lemma~\ref{pro:Gvnormal}.
If $\lefthat T_v=3_{+}^{1+2}{:}\Q_8{:}\frac{(q+1,9)}{3}\in \mathcal{C}_6$, then there is no $q\geq 49$ such that $|T|/|T_v| \leq 30758154560$, a contradiction.
Consequently,  $\lefthat T_v=(q+1)^2{:}\Sy_3 \in \mathcal{C}_2$.
Since $|V(\Ga)|\leq 30758154560$, we have $ 49\leq q \leq 73$.
For these $q$, computation in {\sc Magma}~\cite{Magma} shows that  $ \lefthat T_v $ has no subgroups $K$ and $L$ conjugate in $\lefthat T$ but not in $\lefthat T_v$  such that   $|\lefthat T_v|$ is divisible by $|  K  L|$ and $|\Out(T)|$ is divisible by $|\lefthat T_v|/|  K   L|$.
This contradicts Lemma~\ref{lm:T2at?}(b)(c).

Therefore, we conclude  $n\geq 4$. Now $|T|^{2/3}>30758154560$ by Lemma \ref{lm:T2/3}, which implies that $T_{v}$ is large in $T$.  Applying~\cite[Theorem 7]{AB2015}, we need to consider the following cases:
\begin{enumerate}[\rm (a)]
\item   $T_{v}$ is a $\mathcal{C}_2$-subgroup of type $\GU_{n/k}(q)\wr \Sy_k  $,  and one of the following holds:
\begin{enumerate}
\item[\rm (a.1)] $k=2$;
\item[\rm (a.2)] $k=3$, and either $q\in \{2,3,4\}$ or
\[
(q, (n,q+1)) \in \{ (5,3),(7,1),(7,2),(9,1),(9,2),(13,1),(16,1) \};
\]
\item[\rm (a.3)] either $8\leq n=k \leq 11$ with $q=2$, or $(n,q)=(6,3)$;
\end{enumerate}
\item   $T_{v}$ is a $\mathcal{C}_2$-subgroup of type $\GL_{n/2}(q^2)$;
\item  $T_{v}$ is a $\mathcal{C}_3$-subgroup of type  $\GU_{n/k}(q^k)$, where $k=q=3$ and $n$ is odd;
\item  $T_{v}$ is a $\mathcal{C}_5$-subgroup of type $\GU_{n }(q_0) $, where  $q=q_0^3$;
\item  $T_{v}$ is a $\mathcal{C}_5$-subgroup of type $\Sp_{n }(q) $ or $ \GO_n^{\epsilon}(q)$.
\end{enumerate}

We first deal with Case (a). In this case, the triples $(n,k,q)$ with $|T|/|T_{v}|\leq 30758154560$ are $(6,2,3)$, $(8,2,2)$, and those with  $(n,k)=(4,2)$ and $13 \leq q \leq 19$.
%We construct $\lefthat T_v $ in {\sc Magma}~\cite{Magma}  with the command  {\sf ClassicalMaximals}.
Computation in {\sc Magma}~\cite{Magma} shows that   $ \lefthat T_v $ has no subgroups $K$ and $L$ conjugate in $\lefthat T$ but not in $\lefthat T_v$  such that   $|\lefthat T_v|$ is divisible by $|  K  L|$ and $|\Out(T)|$ is divisible by $|\lefthat T_v|/|\lefthat K \lefthat L|$.
This contradicts Lemma~\ref{lm:T2at?}(b)(c).

For other cases, by~\cite[Tables 2.5,~2.6 ~and~2.8]{BHRD2013}, we see that either $G_{v}^{(\infty)}$ is quasisimple, or  $G_{v}^{(\infty)}/\Z(G_{v}^{(\infty)}) \cong \POm_{4}^{+}(q)\cong \PSL_2(q)^2$ (when $n=4$ and $T_{v}$ is a $\mathcal{C}_5$-subgroup of type  $ \GO_4^{+}(q)$).
Applying Lemma~\ref{lm:Hsiquasi}(b) and Lemma~\ref{lm:ASfac2}, we can rule out Cases (b)--(d), as well as the case that $T_v$ is a  $\mathcal{C}_5$-subgroup of type $\Sp_{n}(q) $ with $(n,q)\neq (4,2^f)$ or of type $\GO_{n}^{\epsilon}(q) $ with $(n,\epsilon)\notin \{(4,+),(8,+) \}$. Since $|T|/|T_v|\leq 30758154560$, it remains to consider the following two possibilities:
\begin{itemize}
\item $T=\PSU_4(q)$ with $q=13$, and  $T_{v}$ is a $\mathcal{C}_5$-subgroup of type $\GO_4^{+}(q)$;
\item  $T=\PSU_4(q)$ with $q \in \{16,\,32,\,64\}$, and  $T_{v}$ is a $\mathcal{C}_5$-subgroup of type $\Sp_4(q)$.
\end{itemize}
For the former,  computation in {\sc Magma}~\cite{Magma} shows that there exists no  factorization $G_v=G_{uv}G_{vw}$ with $|G_{v}|/|G_{uv}|\geq 3$ such that $G_{uv}$ is conjugate to $G_{vw}$ in $G$. For the latter, Lemma~\ref{lm:Hsiquasi}(b) implies that $T_{uv}^{(\infty)}$ and $T_{vw}^{(\infty)}$ are non-conjugate subgroups of $T_v=\Sp_4(q)$ that are isomorphic to $\SL_2(q^2)$. However, computation in {\sc Magma}~\cite{Magma} shows that elements of order $3$ in $T_{uv}^{(\infty)}$ and $T_{vw}^{(\infty)}$ respectively are not conjugate in $\GL_4(q^2)$, contradicting the requirement that $T_{uv}$ and $T_{vw}$ are conjugate in $T$.
\end{proof}

 \subsection{Symplectic groups}
 \ \vspace{1mm}

%Throughout this subsection, suppose Hypothesis~\ref{hy:1} and that $T=\PSp_n(q)$ with $n\geq 4$.

\begin{lemma} \label{lm:pspC1}
Suppose $T=\PSp_n(q)$ with $n\geq 4$ and $|V(\Ga)|\leq 30758154560$.  Then $T_v$ is not a $\mathcal{C}_1$-subgroup.
\end{lemma}

\begin{proof}
Suppose for a contradiction that  $T_v$ is a $\calC_1$-subgroup.
By Lemma~\ref{lm:notparabolic}, $T_v$ is not a parabolic subgroup of $T$, and hence $T_v$ is of type  $\Sp_{m}(q) \perp \Sp_{n-m}(q)$. In particular, we have $n\geq 6$. Let $k=n-m$ and assume without loss of generality that $m<k$.
%Let $k=n-m$ and assume without loss of generality that $m<k$.
Now $\lefthat T_v=\Sp_{m}(q) \times \Sp_{k}(q)$.
Let $N \unlhd G_{v}$ such that $G_{v}/N$ is almost simple with socle $\PSp_k(q)$.
%Then $\pi(N)\subseteq \pi(\SU_m(q)) \cup \pi(q+1) \cup \pi(f)$.
By similar argument as for Case~(e) of Lemma~\ref{lm:G2}, the factors $G_{uv}N/N$ and $G_{vw}N/N$ of the factorization $G_{v}/N=(G_{uv}N/N)(G_{vw}N/N)$ are core-free.
By Lemma~\ref{lm:notsmallT} and the assumption $|V(\Ga)|\leq 30758154560$, one of the following occurs:
\begin{enumerate}[\rm (a)]
\item $n\geq 8$, $m=2$ and $q\in \{2,3\}$;
\item $n=6$,   $k=4$ and $7\leq q \leq 19$;
\item  $n=8$,  $k=6$ and $4\leq q \leq 7$;
\item $(n,k,q)=(10,8,4)$, $(10,6,2)$ or $(12,8,2)$.
\end{enumerate}

Case (a) is ruled out by Lemma~\ref{lm:Hsiquasi}(b).
Take $r\in \ppd(p,kf) $.
Then $|N|_r=1$, which together with $|G_{uv}|_r^2=|G_{vw}|_r^2\geq|G_v|_r$ implies that $r$ divides both  $ |G_{uv}N/N|$ and $ |G_{vw}N/N|$.
By~\cite[Theorem A]{LPS1990}, for those $(k,q)$ in (b)--(d), if $q$ is odd, then there exists no factorization $G_{v}/N=(G_{uv}N/N)(G_{vw}N/N)$ with two  core-free factors, a contradiction.
%Therefore, the triples $(n,k,q)$ are as follows:
%\[
% (6,4,8)\,(6,4,16)\,(8,6,4)\, (10,8,4)\,  (10,8,4)\,(10,6,2),\,(12,8,2).
%\]

First suppose that  (b) occurs. Then $(n,k,q)=(6,4,8)$ or $(6,4,16)$.
Computation in {\sc Magma}~\cite{Magma} shows that both $|G_{uv}N/N|$ and $|G_{vw}N/N|$ are divisible by $|\Sp_2(q^2)|$, while for candidates of $T_{uv}$ and $T_{vw}$ with order divisible by $|\Sp_2(q^2)|$ and index at least $3$ in $T_v$, there is no integer $t$ dividing $ |\Out(T)|$ such that $|T_{v}| =t|T_{uv}T_{vw}|$. This contradicts Lemma~\ref{lm:T2at?}(c).

Next suppose that (c) happens. Then $(n,k,q)=(8,6,4)$. Computation in {\sc Magma}~\cite{Magma} shows that $|G_{uv}N/N|_{17}=17$ and $|G_{vw}N/N|_{17}=1$ (interchanging $G_{uv}$ and $G_{vw}$ if necessary), which together with $|N|_{17}=1$ implies $|G_{uv}|_{17}=17$ and $|G_{vw}|_{17}=1$, contradicting $G_{uv}\cong G_{vw}$.

Finally, suppose that (d) appears. If  $(n,k,q)=(10,8,4)$ or $(n,k,q )=(12,8,2)$, then taking $r=13$ or $7$ respectively, {\sc Magma}~\cite{Magma} computation shows that $|G_{uv}N/N|_r=r$ and $|G_{vw}N/N|_r=1$, and so we conclude from $|N|_r=1$ that $|G_{uv}|_r=r$ and  $|G_{vw}|_r=1$, a contradiction. Now $(n,k,q)=(10,6,2)$. Computation  in {\sc Magma}~\cite{Magma} shows  that either $ |G_{uv}N/N|_{7}=7$ and $ |G_{vw}N/N|_{7}=1$, or $G_{uv}N/N$ and $G_{vw}N/N$ are almost simple groups with distinct socles in $ \{ \A_7,\A_8,\PSL_2(8),\PSU_3(3)\}$.
For the former, it follows from $|N|_7=1$ that  $|G_{uv}|_{7}=7$ and  $|G_{vw}|_{7}=1$, a contradiction.
For the latter, since $N$ has no section in $\{ \A_7,\A_8,\PSL_2(8),\PSU_3(3)\} $, we conclude  $G_{uv} \not\cong G_{vw}$, a contradiction.
\end{proof}

\begin{lemma}\label{lm:psp4q}
Suppose  $|V(\Ga)|\leq 30758154560$. Then  $T \neq \PSp_n(q)$ for $n\geq 4$.
 \end{lemma}
 \begin{proof}
Suppose for a contradiction  that $T=\PSp_n(q)$ with $n\geq 4$. By Lemmas~\ref{lm:notparabolic},~\ref{lm:notsmallT},~\ref{lm:C9} and~\ref{lm:pspC1}, we see that $T_v\notin\{ \mathcal{C}_1,\mathcal{C}_9\}$, and $(n,q)$ is not a pair such that $n=4$ with $q\leq37$, $n=6$ with $q\leq5$, or $(n,q)=(8,2)$.

First suppose that $n=4$. Then $\lefthat T_v$ lies in~\cite[Tables 8.12~and~8.14]{BHRD2013}. Apply Lemma~\ref{lm:Hsiquasi},  we can  rule out the following candidates for $\lefthat T_v$:
\[
\GL_2(q).2,\, \GU_2(q).2,\, \Sp_4(q^{1/r}).(2,r) \text { (with $q$ odd)},\, \Sp_2(q^2){:}2,\, {}^2\B_2(q).
\]
For the remaining candidates of $\lefthat T_v$, as $|T|/|T_v|\leq 30758154560$ and $q\geq 41$, we have either
\begin{itemize}
\item $\lefthat T_v=\Sp_4(q^{1/2}) \in \mathcal{C}_3$ with $q=64$; or
\item $\lefthat T_v =\Sp_2(q)^2{:}2\in \mathcal{C}_2$ or $\mathcal{C}_8$.  %with $41\leq q \leq 471$.
\end{itemize}
For the first candidate, by Lemma~\ref{lm:Hsiquasi} and Lemma~\ref{lm:ASfac}, we derive that $T_{uv}^{(\infty)}\cong T_{vw}^{(\infty)} \cong \Sp_2(8^2)$, but computation in {\sc Magma}~\cite{Magma} shows that $T_{uv}^{(\infty)}$ is not conjugate to $ T_{vw}^{(\infty)} $ in $T$, contradicting the arc-transitivity of $T$.
  Now consider the second candidate. If $q$ is odd, then Lemma~\ref{lm:Omeganonsingular} implies that all $T$-suborbits are self-paired (noticing that $\PSp_4(q)\cong \POm_5(q)$), a contradiction. If $q$ is even, then assume without loss of generality that $T_v$ is a $\mathcal{C}_8$-subgroup $\SO_4^+(q)$ (applying a graph automorphism of order $2$ of $T$ if necessary) and we see from Lemma~\ref{lm:SpOmq} that all $T$-suborbits are self-paired, still a contradiction.

Consequently, we have $n\geq 6$.
 By Lemma \ref{lm:T2/3}, $|T|^{2/3}> 30758154560$  and so $T_{v}$ is large in $T$.
Applying~\cite[Theorem 7]{AB2015},   one of  the following holds:
\begin{enumerate}[\rm (a)]
\item $T_{v}$ is a $\mathcal{C}_8$-subgroup of type $\GO^{\pm}_n(q)$;
\item   $T_{v}$ is a $\mathcal{C}_2$-subgroup of type $\Sp_{n/k}(q)\wr \Sy_k  $, where $k\leq 3$, or $(n,k)=(8,4)$, or $(n,k,q)=(10,5,3)$;
\item   $T_{v}$ is a $\mathcal{C}_2$-subgroup of type $\GL_{n/2}(q)$;
\item  $T_{v}$ is a $\mathcal{C}_3$-subgroup of type $\Sp_{n/2}(q^2),\Sp_{n/3}(q^3)$ or $\GU_{n/2}(q)$;
\item  $T_{v}$ is a $\mathcal{C}_5$-subgroup of type $\Sp_{n }(q_0) $ with $q=q_0^2$;
\item $T_{v}$ is a $\mathcal{C}_6$-subgroup and $(T,T_{v})=(\PSp_8(3),2^6{:}\Omega_6^{-}(2))$.
\end{enumerate}

For Case~(a), since $q$ is even, Lemma~\ref{lm:SpOmq} shows that all $T$-suborbits are self-paired, a contradiction.

Suppose that (b) happens.
The triples $(n,k,q)$ satisfying $|T|/|T_{v}|\leq 30758154560$ are  $  (6,3,7)$, $(6,3,8)$,  $(8,2,3) $, $(8,2,4) $  and  $(8,4,3) $.
Computation in {\sc Magma}~\cite{Magma} shows that  $ \lefthat T_v $ has no subgroups $K$ and $L$ conjugate in $\lefthat T$ but not in $\lefthat T_v$  such that   $|\lefthat T_v|$ is divisible by $|  K  L|$ and $|\Out(T)|$ is divisible by $|\lefthat T_v|/|  K   L|$.
This contradicts Lemma~\ref{lm:T2at?}(b)(c).

Suppose that  (c), (d) or (e) occurs. From~\cite[Tables 2.5,~2.6,~and~2.8]{BHRD2013} we see that $G_{v}^{(\infty)}$ is quasisimple.  By Lemma \ref{lm:Hsiquasi}(b) and  the condition $|V(\Ga)| \leq 30758154560$,  it remains to consider the case that $T=\PSp_8(4)$ and $T_{v}=\PSp_4(16).2 \in \mathcal{C}_3$. In this case,  by Lemma~\ref{lm:Hsiquasi}(a) and Lemma~\ref{lm:ASfac}, we see that $\Ga$ is $(T,2)$-arc-transitive and $T_{uv}\cong T_{vw}\cong \Sp_2(16^2).[4]$. However, {\sc Magma}~\cite{Magma} computation  shows that $T_{uv}$ is not conjugate to $ T_{vw}$ in $T$, a contradiction.

For Case (f), computation in {\sc Magma}~\cite{Magma} shows that   $  T_v $ has no subgroups $K$ and $L$ conjugate in $ T$ but not in $  T_v$  such that   $|  T_v|$ is divisible by $|  K  L|$ and $|\Out(T)|$ is divisible by $| T_v|/|  K   L|$.
This contradicts Lemma~\ref{lm:T2at?}(b)(c).
 \end{proof}

 \subsection{Orthogonal groups}
 \ \vspace{1mm}

\begin{lemma} \label{lm:poC1}
Suppose $T =\POm^{\epsilon}_n(q)$ with $n\geq 7$ and $|V(\Ga)|\leq 30758154560$.
Then $T_v$ is not a $\mathcal{C}_1$-subgroup.
\end{lemma}

 \begin{proof}
Suppose for a contradiction that  $T_v$ is a $\calC_1$-subgroup.
By Lemma~\ref{lm:notparabolic}, $T_v$ is not a parabolic subgroup of $T$, and hence we see from~\cite[Table~2.2]{BHRD2013} that $T_v$ is of type $\Sp_{n-2}(q)$, or type $\GO_{m}^{\epsilon_1}(q) \perp \GO_{n-m}^{\epsilon_2}(q)$.
If $T_v$ is of type $\Sp_{n-2}(q)$, then $n\geq 8$ and $G_v^{(\infty)}$ is quasisimple with $G_v^{(\infty)}/\Z(G_v^{(\infty)})=\PSp_{n-2}(q)$, contradicting Lemma \ref{lm:Hsiquasi}(b).
Consequently, $T_v$ is of type $\GO_{m}^{\epsilon_1}(q) \perp \GO_{n-m}^{\epsilon_2}(q)$.
Let $k=n-m$ and assume without loss of generality that $m\leq k$.

First suppose that $m\leq 2$. Then $k=n-m\geq 5$ and $G_v^{(\infty)}$ is quasisimple with $G_v^{(\infty)}/\Z(G_v^{(\infty)})=\POm^{\epsilon_2}_{k}(q)$. By Lemma~\ref{lm:Hsiquasi}(b) and Lemma~\ref{lm:notsmallT}, we only need to consider the case $\epsilon_2=+$, $m=1$ and $k=8$. Now $T=\POm_9(q)$ with $q$ odd. Then Lemma~\ref{lm:Omeganonsingular} asserts that all $T$-suborbits are self-paired, a contradiction.

Therefore, we conclude $3\leq m \leq n/2$. By Lemma~\ref{lm:notsmallT} and the assumption   $|V(\Ga)|\leq 30758154560$, one of the following holds:
\begin{itemize}
\item $(\epsilon,\epsilon_1,\epsilon_2)=(\circ,\circ,\pm)$ and $(n,m,q)\in \{  ( 7, 3, 7 ), ( 9, 3, 3 )\}$;
\item $(\epsilon,\epsilon_1,\epsilon_2)=(\circ,\pm,\circ)$ and $(n,m,q)=( 9, 4, 3 )\}$;
\item $(\epsilon,\epsilon_1,\epsilon_2)=(\pm,\circ,\circ)$ and $(n,m,q)=( 10, 3, 3 )$;
\item $(\epsilon,\epsilon_1,\epsilon_2)\in \{ (+,+,+),(+,-,-)  \}$ and $(n,m,q)\in \{  ( 12, 4, 2 ) \}$;
\item $(\epsilon,\epsilon_1,\epsilon_2)\in \{(-,-,+), (-,+,-)\}$ and  $(n,m,q)\in \{  ( 8, 4, 4 ),   ( 12, 4, 2 )  \}$.
\end{itemize}
For the above candidates, computation in {\sc Magma}~\cite{Magma} shows that  $ \lefthat T_v $ has no subgroups $K$ and $L$ conjugate in $\lefthat T$ but not in $\lefthat T_v$  such that   $|\lefthat T_v|$ is divisible by $|  K  L|$ and $|\Out(T)|$ is divisible by $|\lefthat T_v|/|  K   L|$. This contradicts Lemma~\ref{lm:T2at?}(b)(c).
\end{proof}

 \begin{lemma} \label{lm:pso}
Suppose $|V(\Ga)|\leq 30758154560$.
Then $T \neq  \POm^{\epsilon}_n(q)$ for $n\geq 7$.
\end{lemma}

 \begin{proof}
Suppose for a contradiction  that $T=\POm^{\epsilon}_n(q)$ with $n\geq 7$.
By Lemmas~\ref{lm:notparabolic},~\ref{lm:notsmallT},~\ref{lm:C9} and~\ref{lm:poC1}, we have $T_v\notin\{ \mathcal{C}_1,\mathcal{C}_9\}$, and $(n,q)$ is not a pair such that $n=7$ with $q\leq5$, $n=8$ with $q\leq3$, or $(n,q)=(10,2)$.
It follows that $|T|^{2/3}> 30758154560$ by Lemma \ref{lm:T2/3},  and so $T_{v}$ is large in $T$.
Applying~\cite[Theorem 7]{AB2015},   we need to consider the following cases:
 \begin{enumerate}[\rm (a)]
\item $T_{v}$ is a $\mathcal{C}_2$-subgroup of type $\GO^{\sigma}_{n/k}(q)\wr \Sy_k  $, and one of the following holds:
\begin{enumerate}
\item[\rm (a.1)]  $k = 2$,
\item[\rm (a.2)]  $(n, k, q, \epsilon , \sigma ) = (12, 3, 2,-, -), (10, 5, 2, -, -)$ or $(8, 4, 2, +, -)$,
\item[\rm (a.3)] $n = k$, $9 \leq n\leq 13$ and $q = 3$;
\end{enumerate}
 \item $T_{v}$ is a $\mathcal{C}_2$-subgroup of type $\GL_{n/2} (q)$;
 \item $T_{v}$ is a $\mathcal{C}_3$-subgroup of type $\GO^{\sigma}_{n/2} (q^2)$ or $\GU_{n/2} (q)$;
 \item  $T_{v}$ is a $\mathcal{C}_4$-subgroup of type $\Sp_{n/2}(q) \otimes \Sp_{2}(q)$ and $(n, \epsilon) = (12, +)$ or $(8, +)$;
 \item $T_{v}$ is a $\mathcal{C}_5$-subgroup of type $\GO^{\sigma}_{n} (q_0)$ with $q=q_0^2$;
  %\item $T_{v}$ is a $\mathcal{C}_6$ -subgroup  with $(T,M)=(\POm_{8}^{+}(3),2^6.\Omega^{+}_6(2))$.
   \item $T_{v}$ is a $\mathcal{C}_7$-subgroup   with $T=\POm_8^{+}(q)$, $T_{v}=\Sp_2(q)\wr \Sy_3$ and $q\leq 2^7$  even;
 \item $T_{v}$ lies in~\cite[Table~3]{AB2015} and so $(T,T_{v})$ is one of the following:
\[ \begin{split}
 &  (\POm_{8}^{+}(q),\G_{2}(q)),\, (\POm_{8}^{+}(q),\GL_3(q)\times \GL_1(q)),\, (\POm_{8}^{+}(q),\GU_3(q)\times \GU_1(q)).
\end{split}\]
\end{enumerate}

Suppose that (a) happens. Since $V(\Ga) \leq 30758154560$, it follows that $(n,q,\epsilon,\sigma,k)=(8,4,+,\pm,2)$ or $(12,2 ,+,-,2)$.
Computation in {\sc Magma}~\cite{Magma} shows that  $ \lefthat T_v $ has no subgroups $K$ and $L$ conjugate in $\lefthat T$ but not in $\lefthat T_v$  such that   $|\lefthat T_v|$ is divisible by $|  K  L|$ and $|\Out(T)|$ is divisible by $|\lefthat T_v|/|  K   L|$.
This contradicts Lemma~\ref{lm:T2at?}(b)(c).

Suppose that (b), (c), (e) or (g)  happens. Then
% from~\cite[Table 2.4,2.5 \& 2.8]{BHRD2013}, we see
 $G_{v}^{(\infty)}$ is quasisimple except for the case  $T=\POm^{+}_8(q)$ and  $T_{v}$ is  a $\mathcal{C}_3$-subgroup of type $\GO^{+}_{4} (q^2)$.
For this exception,   $T_{v}$ is conjugate to a $\mathcal{C}_2$-subgroup of type $\GO^{-}_4(q) \wr \Sy_2$  by a graph automorphism of order $3$ of $T$ (see~\cite[Table 8.50]{BHRD2013}), which has been ruled out in Case (a).
Now $G_{v}^{(\infty)}$ is quasisimple.
By Lemma~\ref{lm:Hsiquasi}, either
\begin{itemize}
\item $n=8$, and $T_{v}$ is a $\mathcal{C}_3$-subgroup of type $\GO^{+}_{8} (q^2)$; or
\item $n=8$, and $T_{v}$ is a $\mathcal{C}_5$-subgroup of type $\GO^{\sigma}_{8} (q_0)$ with $q=q_0^2$.
\end{itemize}
 Since $|T|/|T_{v}|\leq 30758154560$, the only candidate is $T=\POm_8^{+}(4)$ with $T_{v}=\POm_8^{+}(2) \in \mathcal{C}_5$.
By Lemma~\ref{lm:Hsiquasi}(a) and Lemma~\ref{lm:ASfac}, we conclude that $\Ga$ is $(T,2)$-arc-transitive and $T_{uv} \cong T_{vw} \cong \Sp_6(2)$. However,  computation in {\sc Magma}~\cite{Magma} shows that $T_{uv} $ and $ T_{vw} $  are not conjugate in $T$, a contradiction.

Suppose that (d)  occurs. Then the only candidate for $(n,q)$ with $|T|/|T_{v}|\leq 30758154560$ is $(8,4)$. However, from~\cite[Table 8.50]{BHRD2013} we see that $ T=\POm^{+}_8(4)$ has no  $\mathcal{C}_4$-subgroup $T_v$  of type $\Sp_{4}(q) \otimes \Sp_{2}(q)$ such that $G_v$ is maximal in $G$.

Finally, Case (f) is impossible since there is no $q\geq 4$ such that $|T|/|T_{v}|\leq 30758154560$.
 \end{proof}

\end{document}